\documentclass[final,a4paper]{siamltex}

\usepackage[english]{babel}
\usepackage{amsmath,amssymb,amsfonts,graphicx,subfigure}
\usepackage[ruled]{algorithm2e} 
\usepackage{myshortcuts}
\usepackage{soul}
\usepackage{hyperref} 

\usetikzlibrary{external}
\usetikzlibrary{decorations.pathreplacing}

\usepackage{pgfplots}			
\pgfplotsset{compat=1.3}		
\newlength\figureheight 		
\newlength\figurewidth 			
\newcommand{\eigplot}{asterisk}
\newcommand{\shiftplot}{square}
\newcommand{\QNoneplot}{o}
\newcommand{\QNoneplotcolor}{blue}
\newcommand{\QNtwoplot}{diamond}
\newcommand{\QNtwoplotcolor}{red}
\newcommand{\QNthreeplot}{+}
\newcommand{\QNthreeplotcolor}{black}
\newcommand{\QNfourplot}{asterisk}
\newcommand{\QNfourplotcolor}{green}

\graphicspath{{gfx/}}

\iftrue 
  \usepackage{tikz}
  \usepackage{pgfplots}
  \pgfplotsset{
    compat=newest,
    tick label style={font=\scriptsize},
    label style={font=\scriptsize},
    legend style={font=\scriptsize}
  }
  \usetikzlibrary{decorations.pathreplacing}
  \iffalse 
     \usepgfplotslibrary{external} 
  \else
     \renewcommand{\tikzsetnextfilename}[1]{}
  \fi
\else
  \usepackage{tikzexternal}
  \tikzsetexternalprefix{gfx_tmp/}
  
\fi

\newcommand{\ej}[1]{#1}
\newcommand{\gm}[1]{#1}
\newcommand{\ak}[1]{#1}

\title{Disguised and new Quasi-Newton methods for nonlinear eigenvalue problems}
\author{E. Jarlebring,  A. Koskela, G. Mele}
\date{\today}

\begin{document}

\maketitle

\begin{abstract}
  In this \ej{paper} we take a \ej{quasi-Newton} approach to 
  nonlinear eigenvalue problems (NEPs) of the type \ej{$M(\lambda)v=0$}, where $M:\CC\rightarrow\CC^{n\times n}$ is a holomorphic function.
  We investigate which types of approximations of 
  the \gm{Jacobian matrix} lead to competitive algorithms, and provide
  convergence theory. 
  The convergence analysis is based on theory for quasi-Newton methods
  and Keldysh's theorem for NEPs. We derive new algorithms and
  also show that several well-established methods for NEPs
  can be interpreted as quasi-Newton methods,
and thereby provide insight to their convergence behavior. 
  In particular, we establish quasi-Newton interpretations of
  Neumaier's residual inverse iteration and Ruhe's method of successive linear
  problems. 

\end{abstract}

\section{Introduction}\label{sec:intro}

One of the most common techniques to improve
the convergence \ak{or} efficiency of
Newton's method for nonlinear systems
of equations \ak{is to replace the Jacobian matrix
with a different matrix.}
\gm{Among these quasi-Newton method constructions, 
sometimes called inexact Newton methods,  
the most common variation is to keep the Jacobian matrix
constant. 
The factorization of this matrix can be precomputed before carrying out
the iterations.}
\ak{This is beneficial, e.g., in situations where the problem
stems from a discretization of a PDE, as the resulting system} is often large and the \gm{Jacobian matrix} is sparse
with a structure allowing a sparse LU-factorization to be pre-computed.

In this paper we consider nonlinear eigenvalue problems (NEPs) of the type  
\begin{equation}\label{eq:nep}
   M(\lambda)v=0, \;\;v\neq 0
\end{equation}
where $M:\Omega\rightarrow\CC^{n\times n}$. 
There are various flavors of Newton's method
available in the literature (further discussed below) for this class of NEPs.
Some of these methods
do have the property that the matrix in the linear system
to be solved in every iteration remains constant.
However, these methods are in general not seen
as \gm{Jacobian matrix} modifications of Newton's method, 
but are 
often derived from quite different viewpoints.
In this paper we investigate \ak{methods
resulting} from modifying the \gm{Jacobian matrix} in
various ways, and illustrate differences, similarities and
efficiency of the resulting methods. It turns out
that several well-established approaches for NEPs can
be viewed as quasi-Newton methods.


In the NEP-class \gm{that} we consider in this paper
$\Omega\subset\CC$ is a closed set, $M$ is analytic in $\Omega$
and we suppose $\lambda\in\Omega$.
We call the vector $v$ a (right) eigenvector if it satisfies \eqref{eq:nep}
\ak{and $u$ the left eigenvector if it satisfies}
\begin{equation}\label{eq:nep2}
   u^HM(\lambda)=0,\;\;u\neq 0.
\end{equation}
\ej{We call $(\lambda,v,u)$ an eigentriplet of \eqref{eq:nep}.}
Without loss of generality we phrase the NEP as
\ak{a system of equations}
\begin{equation}\label{eq:augsys}
  F\left(\begin{bmatrix}
    v\\\lambda
  \end{bmatrix}\right):=
\begin{bmatrix}
  M(\lambda)v\\
  c^Hv-1
\end{bmatrix}=0
\end{equation}
which is equivalent to \eqref{eq:nep} 
if $c\in\CC^{n}$ is not orthogonal to the eigenvector $v$. This condition is 
not problematic in practice since $c$ can be chosen freely \ak{
and thus }will generically not be orthogonal to any eigenvector.
The quasi-Newton approach to \eqref{eq:augsys} consists of
generating sequences of approximations $(\mu_1,x_1), (\mu_2,x_2),\ldots$
from the relation
\begin{equation}\label{eq:QN}
\tilde{J}_k\begin{bmatrix}
  x_{k+1}-x_k\\
  \mu_{k+1}-\mu_k
\end{bmatrix}=
-\begin{bmatrix}
  M(\mu_k)x_k\\
  c^Hx_k-1
\end{bmatrix}=-F_k
\end{equation}
where $\tilde{J}_k$ is an approximation of the \gm{Jacobian matrix}
\begin{equation}\label{eq:jacobian}
  \tilde{J}_k\approx J_k=J\left(\begin{bmatrix}
    x_k\\\mu_k
  \end{bmatrix}\right):= 
  \begin{bmatrix}
M(\mu_k) & M'(\mu_k)x_k\\
c^H & 0
  \end{bmatrix}
\end{equation}
The eigenvector and eigenvalue
updates will be denoted $\Delta x_k=x_{k+1}-x_k$ and $\Delta\mu_k=\mu_{k+1}-\mu_k$.


We consider four specific \ak{modifications of the} \gm{Jacobian matrix}, briefly
justified as follows.
%
From a quasi-Newton perspective,
the most common approach consists of
keeping the \gm{Jacobian matrix} constant, i.e., setting
\begin{equation}\label{eq:QN1}
\widetilde{J}_{1,k}:=
\begin{bmatrix}
M(\sigma) & M'(\sigma)x_0\\
 c^H & 0 
\end{bmatrix},
\end{equation}
where $\sigma=\lambda_0$ is the starting value for the eigenvalue and $x_0$ the
starting vector for the eigenvector \ak{approximation}.
In this manuscript we 
refer to this as Quasi-Newton 1 (QN1).
In Section~\ref{sec:QN1}
we show how \eqref{eq:QN}
with \gm{Jacobian matrix} approximation \eqref{eq:QN1} can be
reformulated such that in every iteration we need to solve one linear system
associated with the matrix $M(\sigma)$. 

We shall later show (in Section~\ref{sec:QN1}) that a more accurate
approximation of the \gm{Jacobian matrix} leads to an algorithm
which has the same computational cost per iteration as QN1, i.e.,
it involves the solution of one linear system with the matrix $M(\sigma)$ per iteration.
More precisely, we keep only the (1,1)-block constant by setting
\begin{equation}  \label{eq:QN2}
\widetilde{J}_{2,k}
:=
\begin{bmatrix}
M(\sigma) & M'(\mu_k)x_k\\
 c^H & 0 
\end{bmatrix}.
\end{equation}
The quasi-Newton method \eqref{eq:QN} with \gm{Jacobian matrix} approximation \eqref{eq:QN2} will be refered to as Quasi-Newton 2 (QN2).

We also investigate a method (which we call Quasi-Newton 3)
corresponding to keeping  the $(1,1)$-block constant as in QN2, but also replace the derivative in the
  (1,2)-block. We replace the derivative with a  finite difference involving the future eigenvalue
  approximation $\lambda_{k+1}$, i.e., we set
  \begin{equation}
    \label{eq:QN3}
\widetilde{J}_{3,k}
:=
\begin{bmatrix}
M(\sigma) & M[\mu_{k+1},\mu_k]x_k\\
 c^H & 0 
\end{bmatrix},
\end{equation}
where we use the standard notation \gm{for divided differences}
\begin{align}
 M[\lambda,\mu]  = 
 \begin{cases}
 \frac{M(\lambda)-M(\mu)}{\lambda-\mu} 	& \lambda \ne \mu	\\ 
 M'(\lambda) 								& \lambda = \mu.
 \end{cases}
\end{align}

  Note that the modification of the \gm{Jacobian matrix}
  in \eqref{eq:QN3} makes the iteration \eqref{eq:QN} implicit
in the sense that \ak{the formula \eqref{eq:QN} for the
new approximation $(\lambda_{k+1},v_{k+1})$ involves} $\lambda_{k+1}$ in a nonlinear way. 
 Many implicit variations
  of Newton's method have been considered in the literature, see e.g.,
  \cite{homeier2003modified,kou2007improvements} and references therein. It turns out that certain implicit variations
of Newton's method improve \ak{the convergence and sometimes even increase} the convergence order.
  In contrast to many other implicit Newton methods, our choice of \ak{the} \gm{Jacobian matrix} is done with the goal
\ak{of having} a method whose iterates can be
computed without solving a (computationally
demanding)
  nonlinear system of equations.
This is possible for the specific choice \eqref{eq:QN3}, as we shall illustrate in Section~\ref{sec:elimination}.

  We  consider one more modified \gm{Jacobian matrix} which also leads to an implicit method, but now implicit in the  eigenvector.
  The vector $x_k$ in the (1,2) block is replaced by the future vector $x_{k+1}$, such that
  we obtain
  \begin{equation}\label{eq:QN4}
\widetilde{J}_{4,k}:=
\begin{bmatrix}
M(\mu_k) & M'(\mu_k)x_{k+1}\\
 c^H & 0
\end{bmatrix}.
\end{equation}

  \ak{Our study has the following conclusions and contributions:}
    
\begin{itemize}
\item QN1 and QN2 can be phrased as algorithms only involving one linear solve with $M(\sigma)$ per iteration (Section~\ref{sec:QN1}-\ref{sec:QN2})
\item QN3 is equivalent to Neumaier's residual inverse iteration  \cite{Neumaier:1985:RESINV} (Section~\ref{sec:QN3})
\item QN4 is equivalent to Ruhe's method of successive linear problems \cite{Ruhe:1973:NLEVP} (Section~\ref{sec:QN4})
\item We provide exact characterizations of the convergence of factor for QN1 and QN2,
  and establish that the convergence factor of QN2 and QN3 are identical 
  \gm{(Section~\ref{sec:conv_QN12}-\ref{sec:conv_QN3})}
\item We show how to adapt fundamental theory for inexact Newton method \cite{Dembo:1982:INEXACT} to study QN4  
\gm{(Section~\ref{sec:conv_QN4})}
\item We provide generalizations of convergence rate dependence on eigenvalue clustering analogous to
  methods for linear eigenvalue methods 
  \gm{(Section~\ref{sec:interpretations})}
\end{itemize}
\ej{The} \ak{studied properties of the methods are illustrated in numerical simulations} in
Section~\ref{sec:numerics}.

Newton's method for linear and nonlinear eigenvalue problems has been studied
for decades, and the field is still  under active development, as can be observed
in the  summaries in \cite{Mehrmann:2004:NLEVP,Voss:2013:NEPCHAPTER}.
\ak{The technique of deriving} methods using an augmented system as in \eqref{eq:augsys} was
investigated already in 1950's by Unger \cite{Unger:1950:NICHTLINEARE},
and was key to characterizing
the relationship of inverse iteration (for linear eigenvalue problems)
as described by  Peters and Wilkinson in \cite{Peters:1979:INVERSE}.
Newton's method based on solving the nonlinear equation $\beta(\lambda)=0$
where $M(\lambda)v=\beta(\lambda)e_p$ was presented in \cite{Osborne:1964:NEP} and independently by an essentially equivalent 
\ak{procedure by} Lancaster \cite{lancaster1966lambda}.
In 1970's Ruhe \cite{Ruhe:1973:NLEVP} also pointed out the relevance
of Newton-type methods and how they relate to \ak{the} inverse iteration.
He used the augmented systems of
equations to derive variations of \ak{the inverse iteration for NEPs
as} well as the method of successive linear problems.

Newton-type methods have also more recently received considerable
attention, e.g., in the PhD thesis of Schreiber \cite{Schreiber:2008:PHD},
\ak{ where two-sided generalizations of inverse iteration
methods as well as Jacobi-Davidson methods are developed.
The recent results by Effenberger  and Kressner
\cite{Kressner:2009:BLOCKNEWTON,Effenberger:2013:ROBUST}
contain a generalization of Newton-type methods that allows the computation
of several eigenvalues simultaneously. This block
Newton approach has been successful} in the setting of
continuation methods \cite{beyn2011continuation}.
Variants of Newton methods
where the linear system associated with $M(\sigma)$ is
only solved to some accuracy have been studied
in \cite{Szyld:2012:LOCAL}. \ak{A recent variant}
of the rational Krylov method can also be interpreted in a Newton-setting
\cite{VanBeeumen:2013:RATIONAL}.
There are several convergence results for \ak{the residual inverse iteration
and other Newton type methods
\cite{Szyld:2012:LOCAL,Szyld:2014:INVPAIRS,Jarlebring:2011:RESINVCONV,unger2013convergence} which are mostly presented in a separated fashion}
without using quasi-Newton interpretations and results for quasi-Newton methods.

\section{Explicit reformulations of the quasi-Newton methods}\label{sec:elimination}
The formulations of the QN-iterations above are not \ak{practical in general.}
In fact, it is even questionable to call the formulations of QN3 and QN4
iterative, since the \gm{Jacobian matrix} depends on
quantities \ak{in an implicit way.}
Nevertheless, it turns out that \ak{certain reformulations of QN3 and QN4
allow us to explicitly compute sequences of approximations},
\ej{which satisfy \eqref{eq:QN}}.
Also QN1 and QN2 \ak{have to be reformulated} in order to become practical. We show how 
\ak{to carry out this reformulation to obtain algorithms
which do
not require solving many linear systems for different matrices,
but only for $M(\sigma)$. }

%
%
%
%
%
We first make an observation in common \ak{for all the considered} methods.
The last row of the correction \ak{equation \eqref{eq:QN} is the same for all methods
and for all the choices of $\tilde{J}_k$  as in 
\eqref{eq:QN1}-\eqref{eq:QN4}} \ej{we see} \ak{that $c^H(x_{k+1}-x_k)=-c^Hx_k+1$. }
By induction, this implies that 
\begin{equation}\label{eq:normalized}
c^Hx_1=c^Hx_2=\cdots=1,
\end{equation}
i.e.,
all iterates (except possibly the first iterate) are normalized.

\subsection{The $n$-dimensional form of QN1}\label{sec:QN1}
The iteration \eqref{eq:QN} with \ak{the} \gm{Jacobian matrix} 
approximation \eqref{eq:QN1} can be reformulated as follows.
 By multiplying the first block equation in
\eqref{eq:QN} with $M(\sigma)^{-1}$ from the left yields
\begin{equation}\label{eq:QN1_Delta1}
  \Delta x_k+ \Delta\mu_k M(\sigma)^{-1}M'(\sigma)x_0=-M(\sigma)^{-1}M(\mu_k)x_k=-y_k
\end{equation}
where we \ak{define}
\begin{equation}\label{eq:ykdef}
  y_k:=M(\sigma)^{-1}M(\mu_k)x_k.
\end{equation}
Moreover, by multiplication of \eqref{eq:QN1_Delta1} from the left with $c^H$ and using
the fact that $c^H\Delta x_k=0$ due to \eqref{eq:normalized}
 we have
\begin{equation}\label{eq:Delta_lambda_k}
    \mu_{k+1}-\mu_k=\Delta\mu_k=-\alpha_0 c^Hy_k
 \end{equation}
with
\begin{subequations}
\begin{eqnarray}
  q_0		&:=&	M(\sigma)^{-1}M'(\sigma)x_0\label{eq:q0def}\\
 \alpha_0	&:=&	1/c^Hq_0\label{eq:alphadef}
\end{eqnarray}
\end{subequations}
The above equations can be combined into an algorithm.
As a precomputation we compute $\alpha_0$ in \eqref{eq:alphadef} and $q_0$ in \eqref{eq:q0def},
and in the iteration we compute $y_k$ from \eqref{eq:ykdef},
$\Delta\mu_k$ from \eqref{eq:Delta_lambda_k}
and subsequently update
\begin{equation}\label{eq:vk_update}
  x_{k+1}		=x_k-y_k - \Delta \mu_k  q_0.
\end{equation}
The algorithm is summarized in Algorithm~\ref{alg:QN1}.

\begin{remark}[Properties of Algorithm~\ref{alg:QN1}]
  The advantage of Algorithm~\ref{alg:QN1}
  over the original formulation \eqref{eq:QN}
  consists in the fact that the
  matrix in the linear system to be solved in every step is $M(\sigma)$.
  \gm{Therefore, a pre-factorization can be carried out in the original
  problem size.}
  This is of advantage when the problem
  stems from a partial differential equation and the augmented
  matrix $\tilde{J}_{1,k}$ may be more difficult to factorize.
  Although the \gm{Jacobian matrix} approximation $\tilde{J}_{1,k}\approx J_k$
  is the most common in the context of quasi-Newton approximations,
  it does not appear very competitive in this setting.
%
  \gm{ The other Jacobian matrix approximations, which are more accurate, 
  lead to better methods in terms of convergence and do not in general require more
  computation.}
  \end{remark}

%
\begin{algorithm} 
\caption{Quasi-Newton 1 ($n$-dimensional formulation of \eqref{eq:QN1})}\label{alg:QN1}
\SetKwInOut{Input}{input}\SetKwInOut{Output}{output}
\Input{Initial guess of the eigenpair $\gm{(\mu_0,x_0)}\in\CC\times\CC^{n}$}
\Output{An approximation $\gm{(\mu_k,x_k)}\in\CC\times\CC^{n}$ of  $\gm{(\lambda,v)}\in\CC\times\CC^{n}$}
\BlankLine 
\nl Set $\sigma=\mu_0$ and factorize the matrix $M(\sigma)$ \\
\nl Compute $q_0$ and $\alpha_0$ from \eqref{eq:q0def} with the use of the factorization from Step~1\\
\nl \For{$k = 0,1,2,\dots$}{ 
\nl Compute $y_k$ from \eqref{eq:ykdef} with the use of the factorization from Step~1\\
\nl Compute $\Delta \mu_k$ and $\mu_{k+1}$ from \eqref{eq:Delta_lambda_k}\\
\nl Compute new eigenvector approximation  $x_{k+1}$ from \eqref{eq:vk_update}
}
\end{algorithm}

\subsection{The $n$-dimensional formulation of QN2}\label{sec:QN2}
\ak{A reformulation of \eqref{eq:QN} with the} \gm{Jacobian matrix}
approximation \eqref{eq:QN2} \ak{follows} similar steps
as the derivation in the previous section. 
To simplify the notation we set
\begin{equation}  \label{eq:wdef}
  w_\sigma^H:=c^HM(\sigma)^{-1}.
\end{equation}
The derivation leads up to the formulas
\begin{subequations}
\begin{eqnarray}
    \Delta \mu_k 	&=& 	-\frac{w_\sigma^H M(\mu_k) x_k}{w_\sigma^H M'(\mu_k) x_k} 	\label{eq:lambdak_update2}	\\
    z_k			&=&	\Delta \mu_k M'(\mu_k) x_k + M(\mu_k)x_k	\label{eq:zkdef}	\\
    \mu_{k+1}	&=&	\mu_k + \Delta \mu_k				\\
    x_{k+1}		&=&	x_k - M(\sigma)^{-1} z_k	\label{eq:vk_update2}
\end{eqnarray}
\end{subequations}
  More precisely, formula \eqref{eq:lambdak_update2} stems from left-multiplying
  the first block row in \eqref{eq:QN} with $w_\sigma^H=c^HM(\sigma)^{-1}$
  and \eqref{eq:zkdef} and \eqref{eq:vk_update2} stems from left-multiplying
  the first block row in \eqref{eq:QN} with $M(\sigma)^{-1}$.
  We summarize the resulting method in Algorithm~\ref{alg:QN2}.
  \begin{remark}[Properties of Algorithm~\ref{alg:QN2}]
    Note that similar to Algorithm~\ref{alg:QN1}, Algorithm~\ref{alg:QN2}
    requires only one linear solve with the matrix $M(\sigma)$ per iteration.
    Since Algorithm~\ref{alg:QN2}
    corresponds to a more accurate approximation of the \gm{Jacobian matrix},
    it is expected to converge faster than Algorithm~\ref{alg:QN1}.
    This difference characterized theoretically and computationally in
    Section~\ref{sec:conv_QN12} and Section~\ref{sec:numerics}.

    Algorithm~\ref{alg:QN2} involves the vector $w_\sigma$, which can be computed as
    in \eqref{eq:wdef}, i.e., it would require one additional linear solve with $M(\sigma)^H$.
    This extra linear solve
    can however be avoided by treating $w_\sigma$ as a fixed vector (chosen by the user)
    and then using that 
$c$ is arbitrary such that we can chose
it as  $c^H=w_\sigma^{T}M(\sigma)$. This works rather
    well in practice, but fixing $w_\sigma$ instead of $c$
    may make the convergence factor larger if $\sigma$ is close to the
    eigenvalue, as we shall further illustrate in
    Section~\ref{sec:interpretations}.
  \end{remark}
\begin{algorithm} 
\caption{Quasi-Newton 2 ($n$-dimensional formulation of  \eqref{eq:QN2})}\label{alg:QN2}
\SetKwInOut{Input}{input}\SetKwInOut{Output}{output}
\Input{Initial guess of the eigenpair $\gm{(\mu_0,x_0)}\in\CC\times\CC^{n}$}
\Output{An approximation $\gm{(\mu_k,x_k)}\in\CC\times\CC^{n}$ of  $\gm{(\lambda,v)}\in\CC\times\CC^{n}$}
\BlankLine 
\nl Set $\sigma=\mu_0$ and factorize the matrix $M(\sigma)$ \\
\nl \For{$k = 0,1,2,\dots$}{
  \nl Compute $u:=M(\mu_k)x_k$ and $w:=M'(\mu_k)x_k$\\
  \nl Compute $\Delta\mu_k$ according to \eqref{eq:lambdak_update2} using $u$ and $w$\\
  \nl Compute $z_k=\Delta\mu_ku+w$\\
  \nl Compute new eigenpair approximation $\gm{(\mu_k,x_k)}$ from \eqref{eq:lambdak_update2}
and \eqref{eq:vk_update2} by using the factorization computed in Step~1.
}
\end{algorithm}

\subsection{The explicit formulation of  QN3 is residual inverse iteration}\label{sec:QN3}

Although the Jacobi approximation $J_k\approx\tilde{J}_{3,k}$
in \eqref{eq:QN3} involves eigenvalue approximations not
yet computed, we can proceed with an elimination procedure
similar to Section~\ref{sec:QN1} and Section~\ref{sec:QN2}. 

We multiply the first block in equation \eqref{eq:QN}
from the left with $w_\sigma^H=c^HM(\sigma)^{-1}$.
This gives  an equation which  we can simplify as follows:
\begin{subequations}
\begin{eqnarray}
   c^H\Delta x_k+w_\sigma^HM[\mu_{k+1},\mu_k]x_k\Delta\mu_k&=&-w_\sigma^HM(\mu_k)v_k\\
   w_\sigma^H\frac{M(\mu_{k+1})-M(\mu_k)}{\mu_{k+1}-\mu_k}x_k\Delta\mu_k&=&-w_\sigma^HM(\mu_k)v_k    \\
   w_\sigma^HM(\mu_{k+1})x_k&=&0.\label{eq:rf}
\end{eqnarray}
\end{subequations}
Note that \eqref{eq:rf} is a scalar-valued equation, 
in one unknown variable $\mu_{k+1}$, 
 since $x_k$ can be 
viewed as a known vector.
In fact, if we treat $\lambda$ as a function of $x$,
this is the inverse function of $w^HM(\lambda)x=0$ is, which is commonly known as the Rayleigh functional
or generalized Rayleigh quotient \cite{Voss:2004:ARNOLDI,Werner:1970:RAYLEIGH,Voss:1982:MINIMAX}. 
This function generally exists, at least 
 in a neighborhood of
a simple eigenvalue \cite[Proposition~2.1]{Jarlebring:2011:RESINVCONV}, and it is a computable quantity for many problems.

By multiplying the first block row in \eqref{eq:QN} from the left by $M(\sigma)^{-1}$, we obtain
\begin{subequations}
\begin{eqnarray}
  \Delta x_k+M(\sigma)^{-1}(M(\mu_{k+1})-M(\mu_k))x_k&=&-M(\mu_k)x_k\\
  x_{k+1}&=&x_k-M(\sigma)^{-1}M(\mu_{k+1})x_k\label{eq:QN3_xupdate}
\end{eqnarray}
\end{subequations}
Under the assumption that the Rayleigh functional in \eqref{eq:rf} is computable, the relations
 \eqref{eq:rf} and \eqref{eq:QN3_xupdate}
form an explicit algorithm. In fact, 
this algorithm is already
extensively used in current research, where it is commonly known
as \emph{residual inverse iteration} and it was first introduced
by Neumaier in \cite{Neumaier:1985:RESINV}.
Residual inverse iteration also forms the basis of some recent
state-of-the-art algorithms
for NEPs, most importantly the nonlinear Arnoldi method \cite{Voss:2004:ARNOLDI}.

\begin{theorem} \label{thm:QN3} The Quasi-Newton method \eqref{eq:QN} with the modified 
  \gm{Jacobian matrix} \eqref{eq:QN3} is equivalent
  to residual inverse iteration as described in \cite{Neumaier:1985:RESINV}.
\end{theorem}

\begin{remark}[Relation between quasi Newton variant 2 and residual inverse iteration]\label{rem:resinv}
  Due to the analyticity of $M(\lambda)$, 
  the residual inverse iteration 
\eqref{eq:rf} in \eqref{eq:QN3_xupdate} can also be expressed as
\begin{align*} 
\mu_{k+1} &= \mu_k - 
\frac{w^H M(\mu_k) v_k}{w^H M'(\mu_k) v_k}
 - 
 \sum_{j=2}^{\infty}
 \frac{\Delta \mu_k^j}{j!} 
 \frac{w^H M^{(j)}(\mu_k) v_k}{w^H M'(\mu_k) v_k}, \\
 v_{k+1} &= 
 v_k - 
 M(\sigma)^{-1} 
 \left[ 
 M(\mu_k) + 
 \Delta \mu_k
 M'(\mu_k) 
 \right] v_k 
 - 
 \sum_{j=2}^{\infty} 
 \frac{\Delta \mu_k^j }{j!}
  M(\sigma)^{-1} M^{(j)} (\mu_k) v_k.
\end{align*}
From these formulas we see directly that
for linear eigenvalue problems where $M(\lambda)=A-\lambda I$,
QN2 and residual inverse iteration are equivalent.
Hence, they are both generalizations of the
standard inverse iteration method.
This is consistent with the convergence
analysis in Section~\ref{sec:localconv}
which shows that QN2 and QN3 have the same convergence
factor.

\end{remark}


\subsection{The explicit formulation of QN4 is the method of successive linear problems}\label{sec:QN4}
In the previous subsection we saw that iterates 
satisfying \eqref{eq:QN} with the
modified \gm{Jacobian matrix} \eqref{eq:QN3} can be computed in practice and
\ak{the resulting algorithm} is in fact equivalent to a well-known method.
The Jacobi approximation $J_k\approx \tilde{J}_{4,k}$
in \eqref{eq:QN4}
also involves a quantity which we do not have access to at iteration $k$, the vector $x_{k+1}$.
We now show that similar to QN3, we can 
carry out an elimination
such that the update  can be computed
in an explicit way. This algorithm also turns out to be equivalent to
a well-established method. 

The first block row in \eqref{eq:QN} with approximation \eqref{eq:QN4}
simplifies to 
\begin{subequations}
\begin{eqnarray}
  M(\mu_k)(x_{k+1}-x_k)+\Delta\mu_k M'(\mu_k)x_{k+1}&=&-M(\mu_k)x_k\\
  M(\mu_k)x_{k+1}+\Delta\mu_k M'(\mu_k)x_{k+1}&=&0\label{eq:MSLP}
\end{eqnarray}
\end{subequations}
Since we know that the iterates $x_{1},x_2,\ldots$ are normalized, we directly
identify \eqref{eq:MSLP} as a (linear) generalized eigenvalue problem
where $\Delta\mu_k$ is the eigenvalue. Hence, we can
construct an iteration satisfying \eqref{eq:QN} with Jacobi
approximation \eqref{eq:QN4} by repeatedly solving the generalized eigenvalue
problem \eqref{eq:MSLP} and updating the eigenvalue $\mu_{k+1}=\mu_k+\Delta\mu_k$.
This method is known as the method of successive linear
problems and was studied and used by Ruhe in \cite{Ruhe:1973:NLEVP},
where it was justified directly from a Taylor expansion 
of $M(\lambda)$.
\begin{theorem}\label{thm:QN4}
The Quasi-Newton method \eqref{eq:QN} with
  the modified \gm{Jacobian matrix} \eqref{eq:QN4} is equivalent
  to the method of successive linear problems \cite{Ruhe:1973:NLEVP}.
\end{theorem}
%
%
%
\section{Local convergence analysis}\label{sec:localconv}
\subsection{Convergence factor analysis of QN1 and QN2}\label{sec:conv_QN12}
In order to characterize the convergence of QN1 and QN2 we
will derive a 
first-order result. More precisely, we will describe
the local behavior, if $(x_{k},\mu_k)$ is close to $\gm{(\lambda,v)}$,
we describe the behavior
with a matrix $A\in\CC^{(n+1)\times (n+1)}$ such that
\begin{equation}\label{eq:firstorder}
\begin{bmatrix}
  x_{k+1}-v\\
  \mu_{k+1}-\lambda
\end{bmatrix}=
A\begin{bmatrix}
  x_k-v\\
  \mu_k-\lambda
\end{bmatrix}+
\gm{
\mathcal{O} \left(
\mynorm{ \begin{bmatrix}
  x_k-v\\
  \mu_k-\lambda
\end{bmatrix}
}^2
\right)
}
.
\end{equation}
In general, we have linear convergence, with 
a local convergence
factor given by the spectral radius of $A$.
The explicit form of $A$ for 
our first two quasi-Newton methods are given in
following theorems. 
\begin{theorem}[Local convergence Algorithm~\ref{alg:QN1}]\label{thm:convQN1}
  Suppose the sequence $(\mu_1,x_1)$, $(\mu_2,x_2),\ldots$ is generated by Algorithm~\ref{alg:QN1}
  started with $(\mu_0,x_0)$  
  and suppose the sequence converges to the eigenpair $(\lambda,v)$.
Then, the sequence satisfies \eqref{eq:firstorder} with $A=A_1$ where
\begin{equation} \label{eq:A1def}
A_1=
  \begin{bmatrix}
   (I-\alpha_0 q_0 c^H)   M(\sigma)^{-1}\left( M(\sigma) - M(\lambda) \right)	&
   (I-\alpha_0 q_0 c^H)	  M(\sigma)^{-1} M'(\lambda)v	\\
   \alpha_0 c^H M(\sigma)^{-1}\left( M(\sigma) -  M(\lambda) \right)	&
   \alpha_0 c^H M(\sigma)^{-1} M'(\lambda)v
  \end{bmatrix}.  
\end{equation}
  \end{theorem}
\begin{proof}
  \ej{For notational convenience let
    $\tilde{J}_1:\CC^{n+1}\rightarrow\CC^{n+1}$
denote the function corresponding to 
    $\tilde{J}_{1,k}$ in \eqref{eq:QN1}.
 In this fixed-point setting}, our 
  quasi-Newton method can be expressed as 
  \begin{equation}\label{eq:fixedpoint}
  \gm{
  \begin{bmatrix}
    v_{k+1}\\
    \lambda_{k+1}
  \end{bmatrix}=
  \varphi \left( 
  \begin{bmatrix}
    v_{k}\\
    \lambda_{k}
  \end{bmatrix}
  \right)=
    \begin{bmatrix}
    v_{k}\\
    \lambda_{k}
  \end{bmatrix}-\tilde{J}_{1}
  \left(
  \begin{bmatrix}
    v_{k}\\
    \lambda_{k}
  \end{bmatrix}
  \right)^{-1}
  F\left(\begin{bmatrix}
    v_{k}\\
    \lambda_{k}
  \end{bmatrix} \right).
  }
  \end{equation}
  \gm{
  The $A$-matrix in \eqref{eq:firstorder}, corresponding to
  a fixed point map, is given by the Jacobian of $\varphi$.
  In our case this can be explicitly expressed,}
  \ej{
    by using the structure of the Jacobian matrix
    evaluated in the eigenpair}
    \begin{equation}  \label{eq:Jstar}
      J_*:=J\left(\begin{bmatrix}
        v\\
        \lambda
      \end{bmatrix}\right)=
\begin{bmatrix}
M(\lambda)  & M'(\lambda)v\\
c^H & 0
\end{bmatrix}.
    \end{equation}
  \gm{More precisely, we have}
\gm{
\begin{subequations}\label{eq:fixedpoint_A}
  \begin{eqnarray}\label{eq:fixedpoint_A1}
  A&=&\varphi'\left(
  \begin{bmatrix}
    v\\
    \lambda
  \end{bmatrix} \right)=I-\tilde{J}_{1}\left(\begin{bmatrix}
    v\\
    \lambda
  \end{bmatrix}\right)^{-1}J\left(\begin{bmatrix}
    v\\
    \lambda
  \end{bmatrix}\right)=
  \ej{\tilde{J}_{1}\left(\begin{bmatrix}
    v\\
    \lambda
  \end{bmatrix}\right)^{-1}(\tilde{J_{1,*}}-J_*)=}\\\label{eq:fixedpoint_A2}
  &=&\begin{bmatrix}
    M(\sigma)&M'(\sigma)x_0\\
    c^H&0
  \end{bmatrix}^{-1}%
  \begin{bmatrix}
    M(\sigma)-M(\lambda)& M'(\sigma)x_0-M'(\lambda)v\\
    0 & 0
  \end{bmatrix}
  \end{eqnarray}
  The $A$-matrix in \eqref{eq:A1def} follows from \eqref{eq:fixedpoint_A2} after the application of the Schur complement formula for  $\tilde{J}_{1,*}^{-1}$.
\end{subequations}
}
\end{proof}
\begin{theorem}[Local convergence Algorithm~\ref{alg:QN2}]
  Suppose the sequence $(\mu_1,x_1)$, $(\mu_2,x_2),\ldots$ is generated by Algorithm~\ref{alg:QN1}
  started with $(\mu_0,x_0)$ 
  and suppose the sequence converges to the eigenpair $(\lambda,v)$.
Then, the sequence satisfies \eqref{eq:firstorder} with $A=A_2$ where
\begin{equation}\label{eq:A2}
A_2=\begin{bmatrix}
  (I - \alpha q c^H) M(\sigma)^{-1}
  \left[ M(\sigma) - M(\lambda) \right]	  		&	0	\\
  \alpha c^H  M(\sigma)^{-1}\left( M(\sigma) -  M(\lambda) \right) 	&	0
\end{bmatrix}.
\end{equation}
where $\alpha:= (w_\sigma^H M'(\lambda) v)^{-1}$ and 
$q:= M(\sigma)^{-1} M'(\lambda) v$.
\end{theorem}
\begin{proof}
  The proof follows the same reasoning as in the proof of Theorem~\ref{thm:convQN1}, except
  that the \gm{Jacobian matrix} of the fixed point map \gm{$\varphi$} in \eqref{eq:fixedpoint_A1}
  in this case becomes
  \[
 A=\begin{bmatrix}
    M(\sigma)&M'(\lambda)x\\
    c^H&0
  \end{bmatrix}^{-1}%
  \begin{bmatrix}
    M(\sigma)-M(\lambda)& 0\\
    0 & 0
  \end{bmatrix}.
  \]
  The application of the Schur complement formula directly leads to \eqref{eq:A2}.
  \end{proof}
\subsection{Local convergence of QN3}\label{sec:conv_QN3}
Note that QN3 is not a fixed point iteration in the formulation
\eqref{eq:QN3}.
However, 
since  residual inverse iteration and QN3 are equivalent in
the sense of Theorem~\ref{thm:QN3}, we already have a convergence
factor available in  \cite{Jarlebring:2011:RESINVCONV}.
More surprisingly, the convergence factor for
residual inverse iteration (given in
\cite{Jarlebring:2011:RESINVCONV})
 is identical to the convergence
factor of QN2. 
\begin{corollary}[Convergence factor equivalence QN2 and QN3] \label{cor:QNfactors}
  The non-zero eigenvalues of $A_2$ given in
  \eqref{eq:A2} are the same as the non-zero
eigenvalues of 
\begin{align} \label{eq:EliasConvRate}
  B=(I - v c^H) M(\sigma)^{-1} \left[ M(\sigma) - M(\lambda) + \frac{1}{w_\sigma^H M'(\lambda) v} M'(\lambda) vw_\sigma^HM(\lambda) \right] 
\end{align}
and the convergence factors of QN2 and QN3 are the same.
\end{corollary}
\begin{proof}
  Since QN3 is equivalent to residual inverse iteration
  according to Theorem~\ref{thm:QN3} we can directly use the convergence
  characterization in \cite{Jarlebring:2011:RESINVCONV}. More precisely,
  \cite[Theorem~3.1]{Jarlebring:2011:RESINVCONV} states
  that the convergence factor of residual inverse
  iteration is the largest eigenvalue of the matrix
  $B$ in \eqref{eq:EliasConvRate}.

  It remains to show that the non-zero eigenvalues of $B$
  are the same as the non-zero
  eigenvalues of $A_2$ given in \eqref{eq:A2}.
  Clearly, the non-zero eigenvalues of $A_2$ are the non-zero eigenvalues
  of the (1,1)-block of $A_2$.
  The equivalence is based on the general property that if
  a matrix $C$ and vectors $c$ and $v$ satisfy,
  $c^HC=0$ and $c^Hv=1$, then $C$ and $C(I-vc^H)$ have 
  same non-zero eigenvalues. This can be seen from the fact that
  if $\gamma_iz_i=Cz_i$, then  $\gamma_ic^Hz_i=c^HCz_i=0$ and $c^Hz_i=0$ if $\gamma_i\neq 0$,
  i.e., $C(I-vc^H)z_i=Cz_i-vc^Hz_i=\gamma_iz_i$.
  By using this general property where $C$ is the (1,1)-block
  of \eqref{eq:A2}, we obtain $B=C(I-vc^H)$ where $B$ is given in \eqref{eq:EliasConvRate}.
\end{proof}
\subsection{Local convergence of QN4}\label{sec:conv_QN4}
The quasi-Newton method corresponding to
\eqref{eq:QN4} is different in character
 in comparison to the other
\gm{quasi-Newton methods we have considered. This methos does not involve the computation of a linear system for a}
constant matrix.
\ak{Although some convergence results for the method
of successive linear problems (and therefore also QN4)
are} available in the literature \cite{Jarlebring:2012:CONVFACT},
it is natural in our \ak{quasi-Newton approach} to
characterize the convergence using general results
for quasi-Newton methods.
It turns out \gm{that} \eqref{eq:QN4} is a very accurate approximation
of the \gm{Jacobian matrix} and we can apply 
results for quasi-Newton methods given in \cite{Dembo:1982:INEXACT}.

The characterization in \cite{Dembo:1982:INEXACT} is
mainly based on a quantity which describes
the inexactness the quasi-Newton method by comparing it with a
step involving the exact \gm{Jacobian matrix}.
More precisely, we consider the vector $r_k$ in 
\cite{Dembo:1982:INEXACT}, which in our context becomes
\[
r_k:=J_k
\begin{bmatrix}
  x_{k+1}-x_k\\
  \mu_{k+1}-\mu_k
\end{bmatrix}
+\begin{bmatrix}
M(\mu_k)x_k\\
c^Hx_k-1
\end{bmatrix}=\begin{bmatrix}
(\mu_{k+1}-\mu_k)M'(\mu_k)(x_{k+1}-x_{k})\\
0
\end{bmatrix}.
\]
Various results are given \ej{in \cite{Dembo:1982:INEXACT}}
in terms of the norm of $r_k$.
In particular, the result \cite[Theorem 3.3]{Dembo:1982:INEXACT}
\ak{demonstrates that we have a local} quadratic convergence
if  $\|r_k\|\le \mathcal{O}(\|F_k\|^2)$.
Due to the fact that all iterates are normalized 
as in \eqref{eq:normalized},
this condition can be simplified \ak{in our case and stated as}
\begin{equation}\label{eq:rkQN4}
\gm{
\|r_k\|\le
\mathcal{O}\left(\mynorm{F\left(\begin{bmatrix}x_k\\\mu_k\end{bmatrix}\right)}^2 \right)=
\mathcal{O}(\|M(\mu_k)x_k\|^2).}
\end{equation}
In order to use this result, we now suppose
\ej{that the Jacobian matrix in the eigenpair
denoted $J_*$ and defined by \eqref{eq:Jstar}}
 is invertible.
 \ej{The invertability}
 \gm{condition is satisfied if the eigenpair $(\lambda,v)$ is 
such that $c^Hv\neq 0$ and the eigenvalue $\lambda$ is simple.
}
This implies that the function
$F:\CC^{n+1}\rightarrow \CC^{n+1}$ is invertible in a neighborhood of $(\lambda,v)$.
If $(\mu_k,x_k)$ is in this neighborhood, we have
from the implicit function theorem that
\begin{equation}\label{eq:errexp}
\begin{bmatrix}
  x_k-v\\
 \mu_k-\lambda
\end{bmatrix}=-J_*^{-1}\begin{bmatrix}
M(\mu_k)x_k\\
0
\end{bmatrix}+O(\|M(\mu_k)x_k\|^2).
\end{equation}


\begin{theorem}[Convergence QN4]
  Suppose $(\mu_1,x_1),(\mu_2,x_2),\ldots$ are iterates satisfyng \eqref{eq:QN4}
  and suppose they converge to the eigenpair $(\lambda,v)$.
  If $(\lambda,v)$ is a simple or semi-simple eigenpair,
  then $(\mu_k,x_k)$ converges at least quadratically.
\end{theorem}
\begin{proof}
  From \eqref{eq:rkQN4} and properties of the two-norm we have
\begin{multline*}
\|r_k\|\le |\mu_{k+1}-\mu_k|\|M'(\mu_k)\|\|x_{k+1}-x_k\|
\le\\
\left(\left\|\begin{bmatrix}
  x_{k+1}-v\\
  \mu_{k+1}-\lambda
\end{bmatrix}+
\begin{bmatrix}
  x_{k}-v\\
  \mu_k-\lambda
\end{bmatrix}\right\|\right)
\|M'(\mu_k)\|
\left(\left\|\begin{bmatrix}
  x_{k+1}-v\\
  \mu_{k+1}-\lambda
\end{bmatrix}+
\begin{bmatrix}
  x_{k}-v\\
  \mu_k-\lambda
\end{bmatrix}\right\|\right)\le\\
4\left\|\begin{bmatrix}
  x_{k}-v\\
  \mu_k-\lambda
\end{bmatrix}\right\|^2
\|M'(\mu_k)\|.
\end{multline*}
In the last step we used that, due
to the assumption that $(\mu_1,x_1),(\mu_2,x_2),\ldots$
converges to an eigenpair, we have
\[
\left\|
\begin{bmatrix}
x_{k+1}-v\\
\mu_{k+1}-\lambda
\end{bmatrix}
\right\|\le 
\left\|\begin{bmatrix}x_{k}-v\\\mu_{k}-\lambda\end{bmatrix}\right\|
\]
  for sufficiently large $k$.
By using the expansion \eqref{eq:errexp}
we find directly that $\|r_k\|\le\mathcal{O}(\|M(\mu_k)x_k\|^2)$,
which by \cite[Theorem~3.3]{Dembo:1982:INEXACT} implies quadratic convergence.
\end{proof}






\section{Interpretation of convergence factors}\label{sec:interpretations}

In order to provide further insight to the linearly convergent \ak{quasi-Newton methods}, 
we now make characterizations of the matrix $A$. To this end,
we use results associated
with (what is commonly referred to as) Keldysh's theorem; 
see the general formulation in \cite{mennicken2003non},
 and the more recent descriptions in the \ak{context of NEPs, e.g.} \cite{Beyn:2011:INTEGRAL,Szyld:2012:LOCAL}.
For simple eigenvalues, Keldysh's theorem implies that there exists
a function $R_1(\sigma)$, analytic in a neighborhood of the eigenvalue,
such that 
\begin{equation} \label{eq:Keldysh1}
M(\sigma)^{-1} = \frac{1}{\sigma-\lambda_1} \frac{v_1 u_1^H}{u_1^H M'(\lambda_1)v_1}+R_1(\sigma) 
\quad \textrm{for all} \quad \sigma \in \Omega \setminus \{\lambda_1\}.
\end{equation}
where $(\lambda_1,v_1,u_1)$ is the eigentriplet of a simple eigenvalue.

We first observe that the convergence factor of QN2 and QN3 can be directly analyzed
with Keldysh's theorem, since
\gm{
\begin{subequations}\label{eq:rhoA2_keldysh}
\begin{eqnarray}
\rho(A_2)&\le&\mynorm{(I - v_1 c^H) M(\sigma)^{-1} \left[ M(\sigma) - M(\lambda_1) + \frac{M'(\lambda_1) v_1w_\sigma^HM(\lambda_1)}{w_\sigma^H M'(\lambda_1) v_1}  \right] }\\
&\le&|\lambda_1-\sigma|\|(I - v_1 c^H)R_1(\sigma)\left(M'(\lambda_1)+M'(\lambda_1)v_1c^HR_1(\lambda_1)M(\lambda_1)\right)\|\\
&&\phantom{xxxxxxxxxxxxxxxxxxxxxxxxxxxxxxxxxx}+\mathcal{O}(|\lambda_1-\sigma|^2).
\end{eqnarray}
\end{subequations}
}
In the last inequality we used that
\begin{align*}
\lim_{\sigma \rightarrow \lambda_1}
\frac{1}{\sigma-\lambda_1}
\frac{M'(\lambda_1) v_1}{w_\sigma^H M'(\lambda_1) v} w_\sigma^H M(\lambda_1)
= 
M'(\lambda_1) v_1 c^H 
R_1(\lambda_1)
M(\lambda_1).
\end{align*}
The relationship  \eqref{eq:rhoA2_keldysh}  indicates
that the convergence
factor depends linearly on the shift eigenvalue distance, and
linearly in $M'(\lambda_1)$.

\begin{remark}[Double non-semisimple eigenvalues]
In case the eigenvalue $\lambda_1$ is a double non-semisimple eigenvalue,
there exist so called generalized eigenvectors $\widetilde{v}_1$ and $\widetilde{u}_1$ such that
\begin{subequations}
\begin{eqnarray}
  M'(\lambda_1) v + M(\lambda_1) \widetilde{v}&=&0   \\
  u^*M'(\lambda_1)  + \widetilde{u}^*M(\lambda_1)  &=&0.
\end{eqnarray}
\end{subequations}
where $u$ and $v$ are the left and right eigenvectors corresponding to $\lambda_1$.
According to 
\gm{ \cite[Theorem~2.6]{Beyn:2011:INTEGRAL} } 
with $L=1$ and $m_1=2$, 
there is a neighborhood $U$ of $\lambda_1$
where we have the expansion
$$
M(\sigma)^{-1} = \frac{1}{ u^* M'(\lambda_1) \widetilde{v} + \tfrac{1}{2} u^* M''(\lambda_1) v}    
\left[\frac{1}{\sigma - \lambda_1} \left(  v \widetilde{u}^* + \widetilde{v} u^* \right) 
+ \frac{1}{(\sigma - \lambda_1)^2} v u^* \right] + R(\sigma),
$$
where $R(\sigma)$ is analytic in $U$.
Then, instead of \eqref{eq:A2} the iteration matrix $A_2$ is of the form
\begin{equation*}
\begin{aligned}
  A_2 = & (I-vc^T) \left( \frac{1}{\sigma - \lambda_1}\frac{\widetilde{v} u^T} {u^* M'(\lambda_1) \widetilde{v} + \tfrac{1}{2} u^* M''(\lambda_1) v}
  + R(\sigma) \right) \left[ M(\sigma) - M(\lambda) \right] \\
&  + \OOO(|\lambda-\sigma|^2),
\end{aligned}
\end{equation*}
where $R(\sigma)$ contains the contribution from all the eigenvalues other than $\lambda_1$. Thus the
iteration matrix $A - vc^T$ contains the factor $\frac{1}{\sigma - \lambda_1}$ and
unlike in the case of a simple eigenvalue, the convergence factor is not
asymptotically proportional to $|\lambda_1-\sigma|$.
\end{remark}

\subsection{Eigenvalue clustering and condition number}

The convergence factor bound \eqref{eq:rhoA2_keldysh}
provides insight on how the convergence depends on the shift 
if the shift-eigenvalue distance is small (consistent
with what was pointed out in \cite{Jarlebring:2011:RESINVCONV}).  We
now show that different insight can be provided by using a
more general form of Keldysh's theorem. 
\gm{This applies}
to the
situation when the shift-eigenvalue distance is not necessarily small.

Inverse iteration for linear eigenvalue problems (with normalization $c^Hv=1$)
has the following property for diagonalizable matrices.
The convergence factor for the eigentriplet $(\lambda_1,v_1,u_1)$
can be bounded in terms of reciprocal eigenvalue-shift distances
weighted with the condition number
\begin{equation}\label{eq:rho_linear}
\rho\left((I-v_1c^H)\sum_{i=2}^n\frac{\sigma-\lambda_1}{\sigma-\lambda_i}\frac{v_iu_i^H}{u_i^Hv_i}\right)\le
\|P_1\||\sigma-\lambda_1|\sum_{i=2}^n\frac{1}{|\sigma-\lambda_i|}\kappa_i
\end{equation}
where $\kappa_i=\|u_i\|\|v_i\|/|u_i^Hv_i|$ is the eigenvalue condition number (following
the standard definition \cite{VanLoan:1987:ESTIMATING})
and $P_1$ is the projector $P_1=I-v_1c^H$.

In order to generalize this property, we use a more general form of Keldysh's theorem.
We let $\Gamma\subset\Omega$ be a simple, closed, piecewise-smooth curve and
denote the eigenvalues in its interior by $\lambda_1,\cdots, \lambda_k\in \operatorname{int}(\Gamma)$. Then, Keldysh's theorem states that 
\begin{equation} \label{eq:Keldysh}
M(\sigma)^{-1} = \sum_{i=1}^k\frac{1}{\sigma-\lambda_i}\frac{v_i u_i^H}{u_i^H M'(\lambda_i)v_i}+R_\Gamma(\sigma) 
\quad \textrm{for all} \quad \sigma \in \Omega\setminus\{\lambda_1,\ldots, \lambda_k \},
\end{equation}
where \ak{$R_\Gamma$ is analytic in $\operatorname{int}(\Gamma)$.}
The following result provides an analogue of the  eigenvalue clustering
property \eqref{eq:rho_linear},
\gm{under the assumption that $R_\Gamma(\sigma)$ is small.}

\begin{corollary}[Eigenvalue clustering] \label{Cor:clustering}
\gm{
Suppose that $\Omega$ is a  closed simply connected domain with boundary $\Gamma$. 
Suppose that $M$ is analytic
  in this domain and 
  that all the eigenvalues are simple.
  Denote the corresponding eigentriplets by  $(\lambda_1,v_1,u_1)$,$\dots$, \\$(\lambda_k,v_k,u_k)$,}
  with normalization $c^Hv_1=\cdots=c^Hv_k=1$.  Then, the convergence
  factor for QN2 and QN3 are bounded by
\gm{
\begin{multline*}
  \rho(A_2)=\rho(A_3)\le \\
  \|P_1\|
  \mynorm{M(\lambda_1)-M(\sigma)+\frac{1}{w_\sigma^H M'(\lambda_1) v_1} M'(\lambda_1) vw^HM(\lambda_1)}\left(\sum_{i=2}^k \frac{1}{|\sigma-\lambda_i|}\kappa_i
  + \|R_\Gamma(\sigma) \| \right)
\end{multline*}
}
where $\kappa_i$ is the eigenvalue condition number for NEPs,
\[ 
\kappa_i:=\frac{\|u_i\|\|v_i\|}{|u_i^HM'(\lambda_i)v_i|}
\]
and $R_\Gamma(\sigma)$ is the remainder-term  in Keldysh's theorem in \eqref{eq:Keldysh}.
\end{corollary}
\begin{proof} The result follows the steps in \eqref{eq:rhoA2_keldysh}
  but instead using the form \eqref{eq:Keldysh} and matrix norm inequalities.
  \end{proof}


\subsection{Characterization of Keldysh's remainder term $R_\Gamma(\sigma)$}
The analysis in the previous subsection indicates that
a dependence on the eigenvalue clustering  similar to the linear
case can be expected  under the condition that $R_\Gamma$ is small. Keldysh's theorem is, in a certain sense,
\ak{a matrix version of the partial fraction expansion of an
analytic function } 
(also known as  Mittag-Leffler's theorem) and
has been characterized in
\cite{gohberg1986interpolation}. See  
\cite{zwart1988partial} for
\ej{a partial fraction expansion characterization for}
delay eigenvalue problems.
A precise \ak{characterization of}
 $R_\Gamma$ and of its norm for the general case 
using \cite{gohberg1986interpolation} 
is somewhat involved.
We take a less ambitious approach and point out
certain situations where it is small or vanishes. 
Although $R_\Gamma$ does not always vanish in the limit when $\Gamma$ encloses $\CC$,
it does vanish \ej{under certain assumptions}.
We characterize
several \gm{of} such situations next, and note that these
results are general for nonlinear eigenvalue problems.
\ak{Therefore they} may be of interest also outside
the scope of quasi-Newton methods.

First we need an explicit formulation of the remainder term $R_\Gamma$.
\begin{lemma} \label{lem:R_rep}
Let $\Gamma\subset\Omega$ be a simple, closed, piecewise-smooth curve and
denote the eigenvalues of \eqref{eq:nep} in the interior of $\Gamma$
by $\lambda_1,\cdots, \lambda_k\in \operatorname{int}(\Gamma)$ and suppose $\lambda_1,\ldots,\lambda_k$ are simple.
 Then, the analytic function
$R_\Gamma(z)$, i.e., the analytic part of $M(z)^{-1}$ in 
$\mathrm{int}(\Gamma)/\{\lambda_1,\ldots, \lambda_k \}$, given by the Keldysh theorem, has the integral representation
\begin{equation} \label{eq:Cauchy_R}
R_\Gamma(z) = \frac{1}{2 \pi \mathrm{i}} \int_\Gamma  \frac{M(\lambda)^{-1}}{\lambda - z} \, \mathrm{d}\lambda.
\end{equation}
\begin{proof}
  We can apply the Cauchy integral formula, since $R_\Gamma$ is analytic,
  and use equation \eqref{eq:Keldysh}
  \begin{eqnarray*}
  R_\Gamma(z) &=&
  \frac{1}{2 \pi \mathrm{i}} \int_\Gamma \frac{R_\Gamma(z)}{\lambda - z}\,\mathrm{d}\lambda
  =   \\
  &=&\frac{1}{2 \pi \mathrm{i}} \left(\int_\Gamma  \frac{M(\lambda)^{-1}}{\lambda - z} \, \mathrm{d}\lambda -
  \sum\limits_{i=1}^k \frac{v_i u_i^H}{u_i^HM'(\lambda_i)v_i}  \int_\Gamma  \frac{1}{(\lambda-\lambda_i)(\lambda - z)} \, \mathrm{d} \lambda\right)
  \end{eqnarray*}
The Cauchy residue theorem implies that
$\int_\Gamma  \frac{1}{(\lambda-\lambda_i)(\lambda - z)} \, \mathrm{d} \lambda=0$,
which proves \eqref{eq:Cauchy_R}.
\end{proof}
\end{lemma}

This leads directly to a sufficient condition for vanishing $R_\Gamma$,
involving  $M(\lambda)^{-1}$ in the limit $\lambda\rightarrow\infty$.
\begin{lemma} \label{lem:M_inv_series}
  Suppose $M(\lambda)$ is analytic in $\CC$ and suppose all eigenvalues are simple.
  Moreover, suppose
\begin{equation} \label{eq:M_inv_limit}
\norm{M(\lambda)^{-1}} \rightarrow 0, \quad \textrm{as} \quad |\lambda| \rightarrow \infty.
\end {equation}
Then, the set of eigenvalues is finite and
 \begin{equation} \label{eq:M_inv_series}
 M(\lambda)^{-1} = \sum\limits_{i=1}^k \frac{1}{\lambda-\lambda_i}\frac{ v_i u_i^H }{u_i^HM'(\lambda_i)v_i}
 \end{equation}
 for all $\lambda \in \mathbb{C} \setminus \{\lambda_1,\lambda_2,\ldots, \lambda_k  \}$, where $\lambda_i$, $1\leq i\leq k$, denote the eigenvalues of $M$.
\begin{proof}
  Since \eqref{eq:M_inv_limit} implies that for every $\varepsilon>0$
  there exists an $R$ such that
  $\sup_{|z|>R}\|M(z)^{-1}\|<\varepsilon$, we have in particular that there exists
  $r$ such that $M(z)^{-1}$ has no poles outside a disk of radius $r$.
  The eigenvalues of \eqref{eq:nep} are roots of the analytic function $\det(M(\lambda))$.
  An analytic function only has a finite number of roots in a compact
  subset of the complex plane, and we therefore only have a finite number of eigenvalues
  in the disk of radius $r$. 
Take $\Gamma$ to be a circle of radius $r$. 
Then, using the representation given by Lemma~\ref{lem:R_rep}, we get the bound
\begin{equation} \label{eq:R_ineq}
\norm{R_\Gamma (z)} \leq  \frac{1}{2 \pi}  \int_\Gamma  
\frac{\norm{M(\lambda)^{-1}}}{\abs{\lambda - z}} \, \mathrm{d}\lambda
\leq   \max_{z \in \Gamma} \norm{M(z)^{-1}}.
\end{equation}
Letting $r \rightarrow \infty$, the claim follows.
 \end{proof}
\end{lemma}


\begin{remark}[Generalizations to higher order multiplicities]
  Lemma~\ref{lem:R_rep} and Lemma~\ref{lem:M_inv_series} have the
  assumption that the eigenvalues are  simple. A generalization
  to higher algebraic and/or geometric multiplicities seems feasible but
  more involved.
  Then, the expression \eqref{eq:Cauchy_R} follows from
  \cite[Corollary~2.8]{Beyn:2011:INTEGRAL}, and the fact that when the contour $\Gamma$ encircles
 $\lambda_i$ and $z$, and $\ell \geq 1$, it holds that
 $$
 \frac{1}{2 \pi \ii} \int_\Gamma \frac{1}{(\lambda - \lambda_i)^\ell (\lambda - z)} \, \dd \lambda = 0
 $$
 by the residue theorem. Then the condition $\norm{M(\lambda)^{-1}} \rightarrow 0$ as $|\lambda| \rightarrow \infty$,
 and the inequality \eqref{eq:R_ineq} imply that $R_\Gamma(z) = 0$. 
\end{remark}

From the above lemmas we conclude the following result which states
that the $R_\Gamma$ vanishes if the
NEP is the sum of a polynomial with leading non-singular coefficient and
a term which decays sufficiently fast.



\begin{theorem} \label{lem:ABF2}
  Suppose $M$ is analytic in $\CC$ and suppose all 
  eigenvalues are simple. Moreover,
  suppose $M(\lambda)$ is of the form
  $$
M(\lambda) = P(\lambda) + F(\lambda),
$$
where 
\begin{enumerate}
\item $P(\lambda) =  \sum_{i=0}^N A_i \lambda^i$ for some matrices $A_0,\ldots,A_N \in \mathbb{C}^{n \times n}$
such that that $A_N$ is non-singular, and 
\item 
$\lambda^{-N} F(\lambda) \rightarrow 0$ as $\abs{\lambda} \rightarrow \infty$.
\end{enumerate}
Then, $R_\Gamma(\lambda)=0$ and the representation \eqref{eq:M_inv_series} holds for $M(\lambda)^{-1}$.

\begin{proof}
We see that
\begin{equation} \label{eq:M_inv}
M(\lambda)^{-1} =  \Big( \sum\limits_{\ell = 0}^N \lambda^\ell A_\ell + F(\lambda) \Big)^{-1} 
= \frac{ A_N^{-1} }{ \lambda^N } \Big( I +  B(\lambda) \Big)^{-1},
\end{equation}
where $B(\lambda) = \big( \sum\limits_{\ell = 0}^{N-1} \lambda^{\ell-N} A_\ell  + \lambda^{-N} F(\lambda) \big) A_N^{-1}$.
By the assumption 2 above, it holds that
\gm{
\begin{equation} \label{eq:B_vanish}
\begin{aligned}
 \norm{B(\lambda)} & =  \,\,\,  \mynorm{  \big(  \sum\limits_{\ell = 0}^{N-1} \lambda^{\ell - N} A_\ell  + \lambda^{-N} F(\lambda) \big) A_N^{-1}}  \\
 \leq  & \,\,\,  \norm{A_N^{-1}}  \Big( \sum\limits_{\ell = 0}^{N-1}  \abs{\lambda}^{\ell - N} \norm{A_\ell}    +  \norm{ \lambda^{-N} F(\lambda) } \Big)  
\rightarrow 0, \quad \textrm{as}
\quad \abs{\lambda} \rightarrow \infty.
\end{aligned}
\end{equation}
}
Using the bound $\norm{(I - A)^{-1}} \leq (1- \norm{A})^{-1}$ for $\norm{A}<1$ (see~\cite[pp.\;351]{Horn:2012:MATAN}),
and choosing $\lambda$ such that $\norm{B(\lambda)} < 1$, we see from \eqref{eq:M_inv} that
$$
\norm{M(\lambda)^{-1}} \leq \frac{\norm{A_N^{-1}}}{\abs{\lambda}^N}   \frac{1}{1 - \norm{ B(\lambda)} }.
$$ 
and therefore $\norm{M(\lambda)^{-1}}  \rightarrow 0, \quad \textrm{as} \quad \abs{\lambda} \rightarrow \infty$.
\end{proof}
\end{theorem}

\medskip



\gm{In addition to the polynomial eigenvalue problem with invertible leading coefficient matrix 
(eventually perturbed in the sense of Theorem~\ref{lem:ABF2})
the conditions of 
Lemma~\ref{lem:M_inv_series}
(and the representation \eqref{eq:M_inv_series}, subsequently) 
hold for several problems in the literature.
For instance the following class of rational eigenvalue problems \ej{are} often encountered in practice \cite{Betcke:2013:NLEVPCOLL}.}
Let $A$, $B$ and $C_i$, $1\leq i \leq k$, be square matrices such that $B$ is invertible,
and let
$$
M(\lambda) = A + \lambda B + \sum\limits_{i=1}^{k} \frac{\lambda}{\sigma_i - \lambda} C_i,
$$
where $\sigma_i$ are given poles. 

\gm{These results hold also }
in the context of \gm{certain} modifications of the
symmetric eigenvalue problem~\cite{Huang:2010:RANKONE} and ~\cite{solov}. 
Let $A,B \in \mathbb{C}^{n \times n}$ \gm{be} such that $B$ is invertible, 
\gm{
and let $s:\mathbb{C} \rightarrow \mathbb{C}$ be a function such that 
}
$s(\lambda) \rightarrow C$, $C$ constant, as $\abs{\lambda} \rightarrow \infty$. Let
$$
M(\lambda) = A - \lambda B + s(\lambda) u u^H.
$$
The condition on the limit behaviour of $s(\lambda)$ holds in particular
for \cite[Example~2]{Huang:2010:RANKONE}.



\begin{remark}[A counterexample]
Although the above theory shows that $R_\Gamma$ does vanish in many situations similar
to the linear case, it is not always zero, as can be seen from the example 
\begin{equation*}
 \begin{aligned}
M(\lambda)
= \begin{bmatrix}
  \lambda - 1 &  (\lambda - 1)(\lambda - 2) f(\lambda) \\
  0 & \lambda - 2
\end{bmatrix}
 \end{aligned}
\end{equation*}
where $f$ is a scalar analytic function. 
Then,
$$
M(\lambda)^{-1} = \begin{bmatrix} (\lambda - 1)^{-1} & f(\lambda) \\ 0 & (\lambda - 2)^{-1} \end{bmatrix}, 
$$
and when we select $\Gamma$ such that encloses the two eigenvalues $\lambda_1=1$ and $\lambda_2=2$, we
have
\begin{equation*}
 \begin{aligned}
 R_\Gamma(z) = \frac{1}{2 \pi \mathrm{i}} \int_\Gamma \frac{M(\lambda)^{-1}}{\lambda - z} \, \mathrm{d} \lambda = 
  \begin{bmatrix} 0 & f(z) \\ 0 & 0 \end{bmatrix} \quad \textrm{ for all } \, z\in\mathbb{C} \backslash \{1,2\}.
 \end{aligned}
\end{equation*}
\end{remark}

\section{Numerical simulations}\label{sec:numerics}

\subsection{Rational eigenvalue problem}
\ej{We illustrate several properties of the results with numerical simulations}\footnote{The simulations are publicly available online: {\tt http://www.math.kth.se/\~{}eliasj/src/qnewton/}}.
We consider the problem 'loaded string' from the NLEVP collection~\cite{Betcke:2013:NLEVPCOLL}.
The problem is of the form
\begin{equation} \label{eq:loaded_string}
M(\lambda) = A - \lambda B + \frac{\lambda}{\lambda - \sigma} C,
\end{equation}
where $A,B,C \in \mathbb{R}^{n \times n}$, and $B$ is invertible. We set $n=20$.
All the eigenvalues are real positive,
and to make the spectrum more clustered in the left end, we multiply the original coeffient matrix $C$ by $n$.
%
The spectrum of the problem is shown in Figure~\ref{fig:spectrum}.

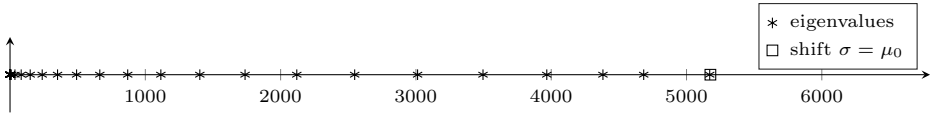
\begin{figure}[h!] 
\begin{center}
\setlength\figureheight{1cm} 
\setlength\figurewidth{10cm}
%
%
\definecolor{mycolor1}{rgb}{1.00000,0.00000,1.00000}%
\begin{tikzpicture}
    \pgfkeys{%
    /pgf/number format/set thousands separator = {}}
\begin{axis}[%
width=1.21\figurewidth,
height=\figureheight,
at={(0\figurewidth,0\figureheight)},
scale only axis,
axis lines=center,
xmin=0,
xmax=6800,
ymin=-1,
ymax=1,
ytick={-100},
axis background/.style={fill=white},
legend style={at={(0.9,1.45)},anchor=north,legend cell align=left,align=left,draw=white!15!black}
]
\addplot [color=black,only marks,mark=\eigplot,mark options={solid}]
  table[row sep=crcr]{%
   5171.410019927617  0\\
   4683.572435686605  0\\
   4382.207980959722  0\\
   3967.197608928568  0\\
   3495.765505391225  0\\
   3012.238497451771  0\\
   2547.367167852586  0\\
   2119.440133202183  0\\
   1737.144097228624  0\\
   1402.746343524835  0\\
   1114.738348177418  0\\
   0869.704917785418  0\\
   0663.507741037136  0\\
   0491.968322581102  0\\
   0351.223776637625  0\\
   0237.885559093553  0\\
   0149.089272126339  0\\
   0082.493155751147  0\\
   0036.263197885961  0\\
   0009.068420939721  0\\
   0001.000000000000  0\\
   0001.000000000000  0\\
   0001.000000000000  0\\
   0001.000000000000  0\\
   0001.000000000000  0\\
   0001.000000000000  0\\
   0001.000000000000  0\\
   0001.000000000000  0\\
   0001.000000000000  0\\
   0001.000000000000  0\\
   0001.000000000000  0\\
   0001.000000000000  0\\
   0001.000000000000  0\\
   0001.000000000000  0\\
   0001.000000000000  0\\
   0001.000000000000  0\\
   0001.000000000000  0\\
   0001.000000000000  0\\
   0001.000000000000  0\\
   0000.046907192055  0\\
};
\addlegendentry{~eigenvalues};
\addplot [color=black,only marks,mark=\shiftplot,mark options={solid}]
  table[row sep=crcr]{%
   5176.410019927617  0\\
};
\addlegendentry{~shift $\sigma=\mu_0$};
\end{axis}
\end{tikzpicture}%
\end{center}
\caption{Spectrum of the rational eigenvalue problem \eqref{eq:loaded_string}.}
\label{fig:spectrum}
\end{figure}

We first place the initial value \gm{$(\mu_0, x_0)$} close to the right-most eigenvalue $\lambda \approx 5170$,
such that $\mu_0 = \lambda + 5.0$ and 
$x_0 = v + a \cdot [ 1 \, \ldots \, 1 ]^T$,
where $v$ is the exact eigenvector corresponding to $\lambda$ and $a$ is a scalar. We set $c=x_0$ and $\sigma = \mu_0$. 
\ej{Since we are here mainly concerned with convergence properties,
  we use computation with high precision arithmetic with a sufficiently high precision, such that round-off errors are not influencing the figures.}
As shown in Figure~\ref{fig:conv1}, the performance of the 
\gm{QN1}
is found to be sensitive to
the distance of $x_0$ from $v$. As expected from Corollary~\ref{cor:QNfactors}, the convergence curves of the Quasi-Newton 2 and the residual
inverse iteration are very close to each other. Figure~\ref{fig:conv_factor} shows the numerically estimated convergence factors
and also the a priori convergence factor estimates.
For the 
\gm{QN1}
the a priori converge factor equals $\rho(A_1)$, i.e., the spectral radius of the matrix $A_1$ given in
Theorem~\ref{thm:convQN1}. For the 
\gm{QN2}
and 
QN3, i.e., the residual inverse iteration, it equals
$\rho(B)$, where $B$ is given in 
\gm{
Corollary~\ref{cor:QNfactors}.}
The estimated convergence factor $\rho_k$ at iteration $k$ is computed by 
$$
\rho_k = \frac{\norm{w_{k} - w_*}}{\norm{w_{k-1} - w_*}},
$$
where $w_k= [ \begin{smallmatrix} v_k \\ \mu_k \end{smallmatrix} ]$ and $w_*$
 denotes the exact solution $w_* =  [ \begin{smallmatrix} v \\ \lambda \end{smallmatrix} ]$.

\begin{figure}[!h]
\centering
\setlength\figureheight{4cm} 
\setlength\figurewidth{4.5cm}
\subfigure[$a=0.15$]{
%
%
\begin{tikzpicture}
\begin{axis}[%
width=0.951\figurewidth,
height=\figureheight,
at={(0\figurewidth,0\figureheight)},
scale only axis,
unbounded coords=jump,
xmin=0,
xmax=30,
xlabel={Iteration $k$},
ymode=log,
ymin=1e-35,
ytick={1e2,1e-2,1e-6,1e-10,1e-14,1e-18,1e-22,1e-26,1e-30,1e-34},
xtick={10,20,30,40,50,60},
ymax=100,
yminorticks=true,
ylabel={Error in $(\mu_k,x_k)$},
axis background/.style={fill=white},
legend style={legend cell align=left,align=left,draw=white!15!black,at={(.95,0.5),anchor=north east}},
mark repeat={2},
grid
]
  \addlegendentry{QN1};
\addplot [color=\QNoneplotcolor,solid,mark=\QNoneplot,mark options={solid}]
  table[row sep=crcr]{%
0	5.02533040961776e+000\\
1	2.18799428596369e+000\\
2	952.892515293142e-003\\
3	415.864249252199e-003\\
4	181.327235871638e-003\\
5	79.0947082133138e-003\\
6	34.4950211023432e-003\\
7	15.0452113066407e-003\\
8	6.56184283783741e-003\\
9	2.86193399514096e-003\\
10	1.24821866882444e-003\\
11	544.406026587571e-006\\
12	237.440422077197e-006\\
13	103.558705580607e-006\\
14	45.1667115884239e-006\\
15	19.6992809687481e-006\\
16	8.59176238995719e-006\\
17	3.74726277953653e-006\\
18	1.63435365016585e-006\\
19	712.816800033591e-009\\
20	310.892192496678e-009\\
21	135.594384723822e-009\\
22	59.1389478632706e-009\\
23	25.7932152717992e-009\\
24	11.2496075438788e-009\\
25	4.90647127784750e-009\\
26	2.13993779837965e-009\\
27	933.325300737205e-012\\
28	407.066092133144e-012\\
29	177.540245864774e-012\\
30      77.4334672203604e-012\\
  };    
\addlegendentry{QN2};
\addplot [color=\QNtwoplotcolor,solid,mark=\QNtwoplot,mark options={solid}]
  table[row sep=crcr]{%
0	5.02533040961776e+000\\
1	2.18799428596369e+000\\
2	2.21291887084610e-003\\
3	2.53547215667873e-006\\
4	5.08750495987077e-009\\
5	27.1718714184824e-012\\
6	233.666402688403e-015\\
7	2.24365114491007e-015\\
8	22.1967889493465e-018\\
9	222.234225460242e-021\\
10	2.23806425800374e-021\\
11	22.6094471564467e-024\\
12	228.804705846100e-027\\
13	2.31780424502738e-027\\
14	23.4934133695308e-030\\
15	238.217388715435e-033\\
16	2.41622761345286e-033\\
17	24.9961578476751e-036\\
18	903.252268580267e-039\\
19	903.202462708485e-039\\
20	903.202381174903e-039\\
21	903.202375020430e-039\\
22	903.202370806081e-039\\
23	903.202369940992e-039\\
24	903.202423938126e-039\\
25	903.202453488558e-039\\
26	903.202460543918e-039\\
27	903.202451418597e-039\\
28	903.202497465179e-039\\
29	903.202489285076e-039\\
  };
  \addlegendentry{QN3};
\addplot [color=\QNthreeplotcolor,solid,mark=\QNthreeplot,mark options={solid}]
  table[row sep=crcr]{%
0	5.02533040961776e+000\\
1	2.18799397937703e+000\\
2	2.21289490925449e-003\\
3	2.53548663018780e-006\\
4	5.08756757907643e-009\\
5	27.1722805881554e-012\\
6	233.669713019620e-015\\
7	2.24368081968998e-015\\
8	22.1970693753265e-018\\
9	222.236953661312e-021\\
10	2.23809124798971e-021\\
11	22.6097168209707e-024\\
12	228.807416193181e-027\\
13	2.31783158269552e-027\\
14	23.4936889323836e-030\\
15	238.217393087407e-033\\
16	2.41622766017416e-033\\
17	24.9961448692269e-036\\
18	903.253404714127e-039\\
19	903.202326101039e-039\\
20	903.202302176741e-039\\
21	903.202214112004e-039\\
22	903.202167814322e-039\\
23	903.202435142170e-039\\
24	903.202261795541e-039\\
25	903.202299452936e-039\\
26	903.202205696467e-039\\
27	903.202180159639e-039\\
28	903.202363166335e-039\\
29	903.202451422177e-039\\
  };
\addlegendentry{QN4};
\addplot [color=\QNfourplotcolor,solid,mark=\QNfourplot,mark options={solid},mark repeat={0}]
  table[row sep=crcr]{%
0	5.02533040e+0\\
1	125.091698e-09\\
2	78.4485351e-24\\
3	601.430837e-39\\
  };
\end{axis}
\end{tikzpicture}%
}\;\;\;%
\subfigure[$a=0.05$]{
%
%
\begin{tikzpicture}
\begin{axis}[%
width=0.951\figurewidth,
height=\figureheight,
at={(0\figurewidth,0\figureheight)},
scale only axis,
unbounded coords=jump,
xmin=0,
xmax=30,
xlabel={Iteration $k$},
ymode=log,
ymin=1e-35,
ytick={1e2,1e-2,1e-6,1e-10,1e-14,1e-18,1e-22,1e-26,1e-30,1e-34},
xtick={10,20,30,40,50,60},
ymax=100,
yminorticks=true,
ylabel={Error in $(\mu_k,x_k)$},
axis background/.style={fill=white},
legend style={legend cell align=left,align=left,draw=white!15!black,at={(.95,0.95),anchor=north east}},
mark repeat={2},
grid
]
  \addlegendentry{QN1};
\addplot [color=\QNoneplotcolor,solid,mark=\QNoneplot,mark options={solid}]
  table[row sep=crcr]{%
0	5.00437872279468e+000\\
1	289.280666613131e-003\\
2	16.3960541061850e-003\\
3	930.514650128360e-006\\
4	52.8043291999938e-006\\
5	2.99653146486432e-006\\
6	170.046556144241e-009\\
7	9.64976821637714e-009\\
8	547.603122908471e-012\\
9	31.0752729274632e-012\\
10	1.76345339657934e-012\\
11	100.072102002777e-015\\
12	5.67887170624092e-015\\
13	322.263480136873e-018\\
14	18.2877437635628e-018\\
15	1.03778923947483e-018\\
16	58.8922569943047e-021\\
17	3.34200606631685e-021\\
18	189.651494395584e-024\\
19	10.7623052181111e-024\\
20	610.737152253292e-027\\
21	34.6579902328119e-027\\
22	1.96676472468310e-027\\
23	111.609573908255e-030\\
24	6.33359759825184e-030\\
25	359.418655539196e-033\\
26	20.3959045320883e-033\\
27	1.15796613694934e-033\\
28	65.3006590442790e-036\\
29	3.91246857420126e-036\\
30      903.202425794579e-039\\
  };    
\addlegendentry{QN2};
\addplot [color=\QNtwoplotcolor,solid,mark=\QNtwoplot,mark options={solid}]
  table[row sep=crcr]{%
0	5.00437872279468e+000\\
1	289.280666613131e-003\\
2	304.390656183161e-006\\
3	402.614152703304e-009\\
4	1.12658688333590e-009\\
5	7.33735922466981e-012\\
6	65.3139138618395e-015\\
7	629.137434342343e-018\\
8	6.22251042322059e-018\\
9	62.2721285275212e-021\\
10	626.924599120654e-024\\
11	6.33199504133639e-024\\
12	64.0702770066578e-027\\
13	648.979709649039e-030\\
14	6.57774580250721e-030\\
15	66.6940148309210e-033\\
16	677.324677394456e-036\\
17	6.93018733675671e-036\\
18	903.207934573470e-039\\
19	903.202778683726e-039\\
20	601.431700235646e-039\\
21	903.203069516315e-039\\
22	903.202920844091e-039\\
23	601.431897403895e-039\\
24	903.203033680353e-039\\
25	903.202833952494e-039\\
26	903.202825069091e-039\\
27	903.202898535326e-039\\
28	903.202914030257e-039\\
29	601.431879783603e-039\\
  };    
\addlegendentry{QN3};
\addplot [color=\QNthreeplotcolor,solid,mark=\QNthreeplot,mark options={solid}]
  table[row sep=crcr]{%
0	5.00437872279468e+000\\
1	289.280517746789e-003\\
2	304.390295416831e-006\\
3	402.614629482266e-009\\
4	1.12658959370624e-009\\
5	7.33737966389976e-012\\
6	65.3140907930335e-015\\
7	629.139075073904e-018\\
8	6.22252621314041e-018\\
9	62.2722837303518e-021\\
10	626.926143738621e-024\\
11	6.33201052890568e-024\\
12	64.0704330029504e-027\\
13	648.981283928785e-030\\
14	6.57776086962746e-030\\
15	66.6955179700638e-033\\
16	675.821952137523e-036\\
17	6.93019120419735e-036\\
18	903.208089947173e-039\\
19	903.202168082542e-039\\
20	903.202237504442e-039\\
21	903.202286004748e-039\\
22	903.202286310726e-039\\
23	601.430739214642e-039\\
24	601.430763150465e-039\\
25	601.430682916213e-039\\
26	903.202278028676e-039\\
27	903.202180385739e-039\\
28	903.202157919388e-039\\
29	903.202188222412e-039\\
  };    
\addlegendentry{QN4};
\addplot [color=\QNfourplotcolor,solid,mark=\QNfourplot,mark options={solid},mark repeat={0}]
  table[row sep=crcr]{%
0	5.00437872279468e+000\\
1	125.091699745373e-009\\
2	78.4485362063254e-024\\
3	601.430821764591e-039\\
  };
\end{axis}
\end{tikzpicture}%
}
\caption{The convergence of the three different methods when a) $c = x_0$ is further from ($a=0.15$) and 
b) closer to ($a=0.05$) the exact eigenvector $v$.}
 \label{fig:conv1}
\end{figure}

\begin{figure}[h!] 
\begin{center}
\setlength\figureheight{4cm} 
\setlength\figurewidth{4.5cm}
%
%
\begin{tikzpicture}
\begin{axis}[%
width=0.951\figurewidth,
height=\figureheight,
at={(0\figurewidth,0\figureheight)},
scale only axis,
unbounded coords=jump,
xmin=0,
xmax=15,
xlabel={iteration k},
ymin=0.43,
    y tick label style={
        /pgf/number format/.cd,
            fixed,
            fixed zerofill,
            precision=3,
        /tikz/.cd
    },
ymax=0.44,
yminorticks=true,
ylabel={Est. Conv. Factor},
axis background/.style={fill=white},
legend style={legend cell align=left,align=left,draw=white!15!black,at={(.95,0.95),anchor=north east}}
]
  \addlegendentry{QN1};
\addplot [color=\QNoneplotcolor,only marks, mark=\QNoneplot,mark options={solid}]
  table[row sep=crcr]{%
0	435.393119977979e-003\\
1	435.509599547901e-003\\
2	436.423040980929e-003\\
3	436.025063942616e-003\\
4	436.198720137691e-003\\
5	436.122995855957e-003\\
6	436.156025589985e-003\\
7	436.141620353390e-003\\
8	436.147903244231e-003\\
9	436.145163006445e-003\\
10	436.146358153965e-003\\
11	436.145836895881e-003\\
12	436.146064240644e-003\\
13	436.145965085162e-003\\
14	436.146008331433e-003\\
15	436.145989469746e-003\\
16	436.145997696195e-003\\
17	436.145994108262e-003\\
18	436.145995673125e-003\\
19	436.145994990616e-003\\
20	436.145995288290e-003\\
21	436.145995158461e-003\\
22	436.145995215085e-003\\
23	436.145995190389e-003\\
24	436.145995201160e-003\\
25	436.145995196462e-003\\
26	436.145995198511e-003\\
27	436.145995197617e-003\\
28	436.145995198007e-003\\
29	436.145995197837e-003\\
  };    
\addlegendentry{Exact};
\addplot [color=black,solid]
  table[row sep=crcr]{%
    0	436.145995197889e-003\\
    15	436.145995197889e-003\\    
  };    
\end{axis}
\end{tikzpicture}%
\;\;\;\;%
%
%
\begin{tikzpicture}
\begin{axis}[%
width=0.951\figurewidth,
height=\figureheight,
at={(0\figurewidth,0\figureheight)},
scale only axis,
unbounded coords=jump,
xmin=0,
xmax=15,
xlabel={iteration k},
ymin=0.0,
ymax=0.02,
    y tick label style={
        /pgf/number format/.cd,
            fixed,
            fixed zerofill,
            precision=3,
        /tikz/.cd
    },
yminorticks=true,
ylabel={Est. Conv. Factor},
axis background/.style={fill=white},
legend style={legend cell align=left,align=left,draw=white!15!black,at={(.95,0.95),anchor=north east}}
]
\addlegendentry{QN2};
\addplot [color=\QNtwoplotcolor,solid,only marks,mark=\QNtwoplot,mark options={solid}]
  table[row sep=crcr]{%
0	435.393119977979e-003\\
1	1.01139152192595e-003\\
2	1.14575920070188e-003\\
3	2.00653158287291e-003\\
4	5.34090318000842e-003\\
5	8.59956972008424e-003\\
6	9.60194156753465e-003\\
7	9.89315518132219e-003\\
8	10.0119988511575e-003\\
9	10.0707451940346e-003\\
10	10.1022332471425e-003\\
11	10.1198717625814e-003\\
12	10.1300549586878e-003\\
13	10.1360645187934e-003\\
14	10.1397521496125e-003\\
15	10.1429523112571e-003\\
16	10.3451172019158e-003\\
17	36.1356442891994e-003\\
18	999.944859400287e-003\\
19	999.999909728344e-003\\
20	999.999993185942e-003\\
21	999.999995333993e-003\\
22	999.999999042197e-003\\
23	1.00000005978409e+000\\
24	1.00000003271740e+000\\
25	1.00000000781149e+000\\
26	999.999989896705e-003\\
27	1.00000005098146e+000\\
28	999.999990943224e-003\\
29	1.00000006269298e+000\\
  };    
\addlegendentry{QN3};
\addplot [color=\QNthreeplotcolor,solid,only marks,mark=\QNthreeplot,mark options={solid}]
  table[row sep=crcr]{%
0	435.393058969720e-003\\
1	1.01138071224700e-003\\
2	1.14577814770336e-003\\
3	2.00654482595304e-003\\
4	5.34091786808032e-003\\
5	8.59956205227317e-003\\
6	9.60193253415594e-003\\
7	9.89314931987235e-003\\
8	10.0119952730491e-003\\
9	10.0707430115360e-003\\
10	10.1022319091230e-003\\
11	10.1198709388948e-003\\
12	10.1300544416733e-003\\
13	10.1360638571771e-003\\
14	10.1396334042309e-003\\
15	10.1429523212337e-003\\
16	10.3451116305097e-003\\
17	36.1357085038395e-003\\
18	999.943450406252e-003\\
19	999.999973511695e-003\\
20	999.999902497219e-003\\
21	999.999948740513e-003\\
22	1.00000029597786e+000\\
23	999.999808075552e-003\\
24	1.00000004169320e+000\\
25	999.999896195493e-003\\
26	999.999971726345e-003\\
27	1.00000020261986e+000\\
28	1.00000009771436e+000\\
29	1.00000004778626e+000\\
  };    
\addlegendentry{Exact};
\addplot [color=black,solid]
  table[row sep=crcr]{%
    0	10.1453289895710e-003\\
15	10.1453289895710e-003\\    
  };    
\end{axis}

\end{tikzpicture}%
\end{center}
\caption{Estimated convergence factors of the three different methods. The solid line is the a priori computed spectral 
radius of $A$ in Corollary~\ref{cor:QNfactors} and the dashed line that of Theorem~\ref{thm:convQN1}.}
\label{fig:conv_factor}
\end{figure}
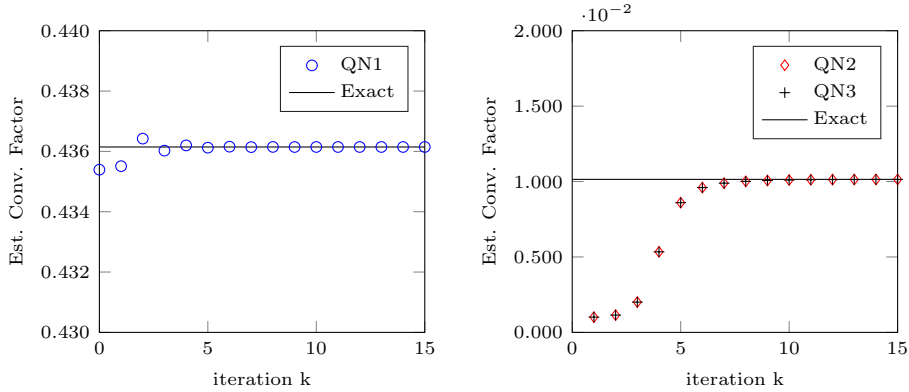
Then, we consider the initial value \gm{$(\mu_0, x_0)$} close to a left-end eigenvalue $\lambda \approx 9.07$,
such that again $\mu_0 = \lambda + 5.0$, 
$x_0 = v + a \cdot [ 1 \, \ldots \, 1 ]^T$ ($a>0$), and $c=x_0$ and $\sigma = \mu_0$. 
As can be expected from the bound given in~\ref{Cor:clustering}, the convergence is now slower since $\lambda$ \gm{is} in a cluster of eigenvalues. 
This is depicted by Figure~\ref{fig:conv2}.
Again, the performance of the 
\gm{QN1}
is found to be sensitive to the distance \ej{of $x_0$} from $v$.

As shown in Figure~\ref{fig:conv2}, the convergence is now slower than for the right-most eigenvalue,
as is expected from the clustering of the spectrum depicted in Figure~\ref{fig:spectrum},

\begin{figure}[!h] 
\centering
\setlength\figureheight{4cm} 
\setlength\figurewidth{4.5cm}
\subfigure[$a=0.2$]{
%
%
\begin{tikzpicture}
\begin{axis}[%
width=0.951\figurewidth,
height=\figureheight,
at={(0\figurewidth,0\figureheight)},
scale only axis,
unbounded coords=jump,
xmin=0,
xmax=30,
xlabel={Iteration $k$},
ymode=log,
ymin=1e-35,
ytick={1e2,1e-2,1e-6,1e-10,1e-14,1e-18,1e-22,1e-26,1e-30,1e-34},
xtick={10,20,30,40,50,60},
ymax=100,
grid,
yminorticks=true,
ylabel={Error in $(\mu_k,x_k)$},
axis background/.style={fill=white},
legend style={legend cell align=left,align=left,draw=white!15!black,at={(.6,0.43),anchor=north east}},
mark repeat={2}
]
  \addlegendentry{QN1};
\addplot [color=\QNoneplotcolor,solid,mark=\QNoneplot,mark options={solid}]
  table[row sep=crcr]{%
0	5.04879498292484e+000\\
1	1.68143091508368e+000\\
2	232.270160610137e-003\\
3	71.0330120456898e-003\\
4	17.4488584696863e-003\\
5	4.45585193951951e-003\\
6	1.12585678464126e-003\\
7	285.126830344826e-006\\
8	72.1628604160613e-006\\
9	18.2664013503437e-006\\
10	4.62354052559091e-006\\
11	1.17030869722075e-006\\
12	296.227299800294e-009\\
13	74.9807863780874e-009\\
14	18.9790657691724e-009\\
15	4.80396337459835e-009\\
16	1.21597470309229e-009\\
17	307.786376994708e-012\\
18	77.9065991747507e-012\\
19	19.7196453437004e-012\\
20	4.99141814151760e-012\\
21	1.26342308035927e-012\\
22	319.796465598223e-015\\
23	80.9465815521604e-015\\
24	20.4891228323064e-015\\
25	5.18618756206329e-015\\
26	1.31272293348207e-015\\
27	332.275198200544e-018\\
28	84.1051866492121e-018\\
29	21.2886259931730e-018\\
30      5.38855705258037e-018\\
  };    
\addlegendentry{QN2};
\addplot [color=\QNtwoplotcolor,solid,mark=\QNtwoplot,mark options={solid}]
  table[row sep=crcr]{%
0	5.04879498292484e+000\\
1	1.68143091508368e+000\\
2	194.238486416436e-003\\
3	37.4798008525940e-003\\
4	4.98092486207973e-003\\
5	873.066563222340e-006\\
6	123.119273465852e-006\\
7	20.5740827057074e-006\\
8	3.00881302799109e-006\\
9	488.436706124765e-009\\
10	73.0501082458484e-009\\
11	11.6499258494957e-009\\
12	1.76659007230461e-009\\
13	278.679069501809e-012\\
14	42.6198317127584e-012\\
15	6.67839281648336e-012\\
16	1.02672361508419e-012\\
17	160.223477725725e-015\\
18	24.7119596576428e-015\\
19	3.84663278387782e-015\\
20	594.460162842848e-018\\
21	92.3890441589772e-018\\
22	14.2952634110700e-018\\
23	2.21959771566512e-018\\
24	343.693855141837e-021\\
25	53.3332773427868e-021\\
26	8.26220900688366e-021\\
27	1.28163830670535e-021\\
28	198.603387538345e-024\\
29	30.8006065512822e-024\\
  };    
\addlegendentry{QN3};
\addplot [color=\QNthreeplotcolor,solid,mark=\QNthreeplot,mark options={solid}]
  table[row sep=crcr]{%
0	5.04879498292484e+000\\
1	1.68455483167612e+000\\
2	194.140826585185e-003\\
3	37.5157069187103e-003\\
4	4.98248396015306e-003\\
5	873.780276095544e-006\\
6	123.173645772631e-006\\
7	20.5893413945418e-006\\
8	3.01034846439905e-006\\
9	488.775628167338e-009\\
10	73.0903819923114e-009\\
11	11.6576609133821e-009\\
12	1.76760787720316e-009\\
13	278.858908562855e-012\\
14	42.6450312854783e-012\\
15	6.68262546509393e-012\\
16	1.02734017418626e-012\\
17	160.323881729446e-015\\
18	24.7269395157173e-015\\
19	3.84902635258731e-015\\
20	594.822579321475e-018\\
21	92.4462828514194e-018\\
22	14.3040091492506e-018\\
23	2.22096914051407e-018\\
24	343.904575756196e-021\\
25	53.3661756414419e-021\\
26	8.26728128126788e-021\\
27	1.28242806726395e-021\\
28	198.725411217915e-024\\
29	30.8195742703153e-024\\
  };    
\addlegendentry{QN4};
\addplot [color=\QNfourplotcolor,solid,mark=\QNfourplot,mark options={solid},mark repeat={0}]
  table[row sep=crcr]{%
0	5.04879498292484e+000\\
1	1.33206550584350e-003\\
2	122.879862216785e-012\\
3	1.04558241342525e-024\\
4	45.9586591057445e-039\\
  };
\end{axis}
\end{tikzpicture}%
}
\subfigure[$a=0.1$]{
%
%
\begin{tikzpicture}
\begin{axis}[%
width=0.951\figurewidth,
height=\figureheight,
at={(0\figurewidth,0\figureheight)},
scale only axis,
unbounded coords=jump,
xmin=0,
xmax=30,
xlabel={Iteration $k$},
ymode=log,
ymin=1e-35,
ytick={1e2,1e-2,1e-6,1e-10,1e-14,1e-18,1e-22,1e-26,1e-30,1e-34},
xtick={10,20,30,40,50,60},
ymax=100,
grid,
yminorticks=true,
ylabel={Error in $(\mu_k,x_k)$},
axis background/.style={fill=white},
legend style={legend cell align=left,align=left,draw=white!15!black,at={(.6,0.43),anchor=north east}},
mark repeat={2}
]
  \addlegendentry{QN1};
\addplot [color=\QNoneplotcolor,solid,mark=\QNoneplot,mark options={solid}]
  table[row sep=crcr]{%
0	5.01665337723374e+000\\
1	861.274281893122e-003\\
2	18.0704042460889e-003\\
3	11.4155871604931e-003\\
4	1.29721816617978e-003\\
5	246.385605502119e-006\\
6	37.9123427071788e-006\\
7	6.31857883231294e-006\\
8	1.02086323508735e-006\\
9	166.926728631933e-009\\
10	27.1688352301241e-009\\
11	4.42990781766605e-009\\
12	721.801260482539e-012\\
13	117.640495671950e-012\\
14	19.1712811505266e-012\\
15	3.12437237442711e-012\\
16	509.175774823696e-015\\
17	82.9803514171870e-015\\
18	13.5232726588200e-015\\
19	2.20388397194245e-015\\
20	359.166232213917e-018\\
21	58.5332082742162e-018\\
22	9.53913811281310e-018\\
23	1.55459029887512e-018\\
24	253.351083735187e-021\\
25	41.2885451986292e-021\\
26	6.72878101616554e-021\\
27	1.09658729232967e-021\\
28	178.710480635913e-024\\
29	29.1243899273947e-024\\
30      4.74639252047743e-024\\
  };    
\addlegendentry{QN2};
\addplot [color=\QNtwoplotcolor,solid,mark=\QNtwoplot,mark options={solid}]
  table[row sep=crcr]{%
0	5.01665337723374e+000\\
1	861.274281893122e-003\\
2	97.5332457148065e-003\\
3	18.7640714474756e-003\\
4	2.49160881448353e-003\\
5	436.810844415749e-006\\
6	61.5859701228028e-006\\
7	10.2930289950323e-006\\
8	1.50510035863861e-006\\
9	244.354339763971e-009\\
10	36.5426781788776e-009\\
11	5.82811549549693e-009\\
12	883.732668666749e-012\\
13	139.413594834003e-012\\
14	21.3206418891127e-012\\
15	3.34095112989207e-012\\
16	513.622674475106e-015\\
17	80.1535391838530e-015\\
18	12.3622950615496e-015\\
19	1.92431552145652e-015\\
20	297.382536621022e-018\\
21	46.2184540829343e-018\\
22	7.15130582448396e-018\\
23	1.11037284631756e-018\\
24	171.935380138194e-021\\
25	26.6804163956134e-021\\
26	4.13323259134328e-021\\
27	641.150031254281e-024\\
28	99.3528734995043e-024\\
29	15.4082518838831e-024\\
30      2.38808680931239e-024\\
  };  
\addlegendentry{QN3};
\addplot [color=\QNthreeplotcolor,solid,mark=\QNthreeplot,mark options={solid}]
  table[row sep=crcr]{%
0	5.01665337723374e+000\\
1	861.559410620604e-003\\
2	97.4935374010995e-003\\
3	18.7667639895928e-003\\
4	2.49143880328284e-003\\
5	436.852940979449e-006\\
6	61.5844289185983e-006\\
7	10.2937685960910e-006\\
8	1.50509612751321e-006\\
9	244.368131524381e-009\\
10	36.5430591418577e-009\\
11	5.82838814127323e-009\\
12	883.748962717335e-012\\
13	139.419278442049e-012\\
14	21.3211390644620e-012\\
15	3.34107488569321e-012\\
16	513.636184615757e-015\\
17	80.1563237166146e-015\\
18	12.3626428472072e-015\\
19	1.92437963979333e-015\\
20	297.391236642577e-018\\
21	46.2199536514940e-018\\
22	7.15151996902024e-018\\
23	1.11040827451998e-018\\
24	171.940601565018e-021\\
25	26.6812588314188e-021\\
26	4.13335918822775e-021\\
27	641.170144820267e-024\\
28	99.3559324937008e-024\\
29	15.4087333236552e-024\\
30    2.38816057176621e-024\\
  };
\addlegendentry{QN4};
\addplot [color=\QNfourplotcolor,solid,mark=\QNfourplot,mark options={solid},mark repeat={0}]
  table[row sep=crcr]{%
0	5.01665337723374e+000\\
1	1.33206550584350e-003\\
2	122.879862216785e-012\\
3	1.04558241342525e-024\\
4	45.9586591057445e-039\\
  };
\end{axis}
\end{tikzpicture}%
}
\caption{The convergence of the three different methods when a) $c = x_0$ is further from ($a=0.2$) and 
b) closer to ($a=0.1$) the exact eigenvector $v$.}
\label{fig:conv2}
\end{figure}

\subsection{Quadratic eigenvalue problem}

In this \ej{section} we illustrate the influence of the eigenvalue clustering on the convergence factor.
The bounds for the convergence factor illustrating this effect were 
derived in Section~\ref{sec:localconv} and can be clearly identified from the following example.
Consider the quadratic eigenvalue problem
\begin{equation} \label{eq:M_circle}
M(\lambda) = \lambda^2 I - \lambda (A_1 + A_2) +  A_1 A_2,
\end{equation}
where $A_1, A_2 \in \mathbb{C}^{10 \times 10}$ are diagonal matrices.
\ej{The set eigenvalues of this problem is the
  uninion of eigenvalues $A_1$ and $A_2$.}
We choose the diagonal elements of $A_1$ and $A_2$ such that eigenvalues of $M(\lambda)$
are $0.1$ and $19$ equally distributed points on the \ak{circle of radius $r$,} 
as illustrated in Figure~\ref{fig:spectrum2} for $r=0.5$.
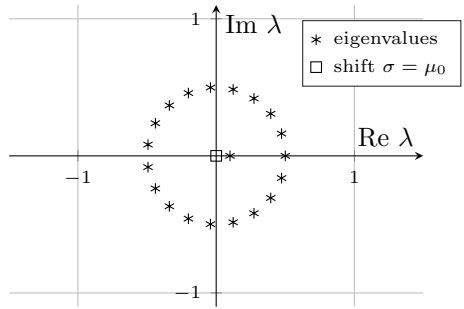
\begin{figure}[h!] 
\begin{center}
\setlength\figureheight{4cm} 
\setlength\figurewidth{4.5cm}
%
%
\definecolor{mycolor1}{rgb}{1.00000,0.00000,1.00000}%
\begin{tikzpicture}
    \pgfkeys{%
    /pgf/number format/set thousands separator = {}}
\begin{axis}[%
width=1.21\figurewidth,
height=\figureheight,
at={(0\figurewidth,0\figureheight)},
scale only axis,
xmin=-0.7,
xmax=0.7,
ymin=-1.1,
ymax=1.1,
axis equal,
axis lines=center, 
grid,
ylabel={Im $\lambda$},
xlabel={Re $\lambda$},
xtick={-1, 0,1},
ytick={-1, 0,1},
axis background/.style={fill=white},
legend style={at={(0.9,0.95)},anchor=north,legend cell align=left,align=left,draw=white!15!black}
]
\addplot [color=black,only marks,mark=\eigplot,mark options={solid}]
  table[row sep=crcr]{%
    500.000000000000e-003  0.00000000000000e+000   \\
    472.908620850317e-003   162.349734602342e-003  \\
    472.908620850317e-003  -162.349734602342e-003  \\
    394.570254698197e-003   307.106356344834e-003  \\
    394.570254698197e-003  -307.106356344834e-003  \\
    273.474079061213e-003   418.583239131264e-003  \\
    273.474079061213e-003  -418.583239131264e-003  \\
    122.742743570400e-003   484.700132969665e-003  \\
    122.742743570399e-003  -484.700132969665e-003  \\
    100.000000000000e-003   0.00000000000000e+000  \\
   -41.2896727361661e-003   498.292246503335e-003  \\
   -41.2896727361664e-003  -498.292246503335e-003  \\
   -200.847712326485e-003   457.886663327529e-003  \\
   -200.847712326485e-003  -457.886663327529e-003  \\
   -338.640785812870e-003   367.861955336566e-003  \\
   -338.640785812871e-003  -367.861955336566e-003  \\
   -439.736875603244e-003   237.973696518537e-003  \\
   -439.736875603245e-003  -237.973696518537e-003  \\
   -493.180651701361e-003   82.2972951403670e-003  \\
   -493.180651701361e-003  -82.2972951403669e-003  \\
};
\addlegendentry{~eigenvalues};
\addplot [color=black,only marks,mark=\shiftplot,mark options={solid}]
  table[row sep=crcr]{%
   0 0\\
};
\addlegendentry{~shift $\sigma=\mu_0$};
\end{axis}
\end{tikzpicture}%
\end{center}
\caption{Spectrum of the quadratic problem \eqref{eq:M_circle} when $r=0.5$.}
\label{fig:spectrum2}
\end{figure}
  
We first place the initial eigenpair approximation \gm{$(\mu_0, x_0)$} to the origin, such that $\mu_0 = 0$ and 
$x_0 = v + 0.1 \cdot [ 1 \, \ldots \, 1 ]^T$,
where $v$ is the exact eigenvector corresponding to the eigevalue 0.1. 
We set $c=x_0$ and $\sigma = \mu_0$. 
Figure~\ref{fig:conv_poly} shows the convergence of the methods and also the numerically estimated convergence factors
for the three first quasi-Newton methods.
\gm{QN1}
is found again to have the slowest convergence, and
the a priori convergence factor estimates computed using Theorem~\ref{thm:convQN1} and \gm{Corollary~\ref{cor:QNfactors}}
are again found to be sharp.

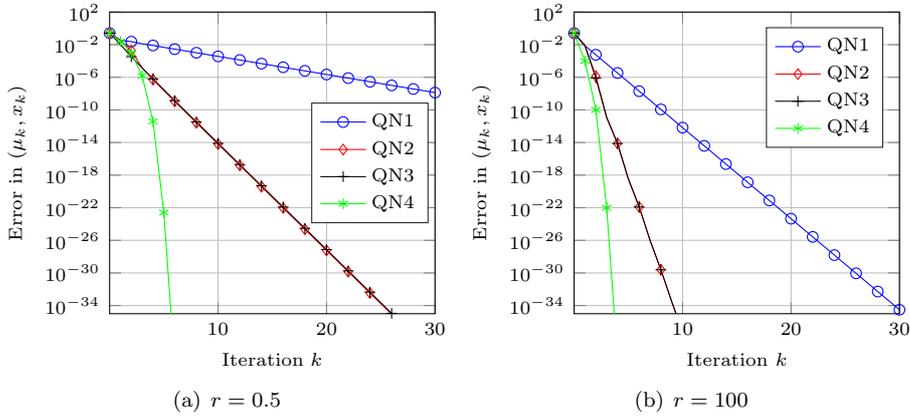
\begin{figure}[!h] 
\centering
\setlength\figureheight{4cm} 
\setlength\figurewidth{4.5cm}
\subfigure[$r=0.5$]{
%
%
\begin{tikzpicture}
\begin{axis}[%
width=0.951\figurewidth,
height=\figureheight,
at={(0\figurewidth,0\figureheight)},
scale only axis,
unbounded coords=jump,
xmin=0,
xmax=30,
xlabel={Iteration $k$},
ymode=log,
ymin=1e-35,
ytick={1e2,1e-2,1e-6,1e-10,1e-14,1e-18,1e-22,1e-26,1e-30,1e-34},
xtick={10,20,30,40,50,60},
ymax=100,
grid,
yminorticks=true,
ylabel={Error in $(\mu_k,x_k)$},
axis background/.style={fill=white},
legend style={legend cell align=left,align=left,draw=white!15!black,at={(.98,0.7),anchor=north east}},
mark repeat={2}
]
  \addlegendentry{QN1};
\addplot [color=\QNoneplotcolor,solid,mark=\QNoneplot,mark options={solid}]
  table[row sep=crcr]{%
0	259.259545168478e-003\\
1	35.3342540410338e-003\\
2	23.4179089216051e-003\\
3	12.8424417039339e-003\\
4	8.02357773373659e-003\\
5	4.70021102052737e-003\\
6	2.84920445761074e-003\\
7	1.69957639204413e-003\\
8	1.02071272559933e-003\\
9	611.612340923818e-006\\
10	366.548647490936e-006\\
11	219.802053884550e-006\\
12	131.706675420227e-006\\
13	78.9725036797795e-006\\
14	47.3276994931955e-006\\
15	28.3739202177032e-006\\
16	17.0063646213289e-006\\
17	10.1947451037437e-006\\
18	6.11076611272128e-006\\
19	3.66304775848537e-006\\
20	2.19570153697555e-006\\
21	1.31617367257903e-006\\
22	788.947493757525e-009\\
23	472.917727814616e-009\\
24	283.479609088273e-009\\
25	169.925522302506e-009\\
26	101.858012596929e-009\\
27	61.0564883970212e-009\\
28	36.5989371430677e-009\\
29	21.9384064767915e-009\\
30      13.1504837290315e-009\\
  };    
\addlegendentry{QN2};
\addplot [color=\QNtwoplotcolor,solid,mark=\QNtwoplot,mark options={solid}]
  table[row sep=crcr]{%
0	259.259545168478e-003\\
1	35.3342540410338e-003\\
2	1.89066694827084e-003\\
3	14.4361074699171e-006\\
4	519.847077243198e-009\\
5	23.7506862947524e-009\\
6	1.12127295146481e-009\\
7	53.8737229688285e-012\\
8	2.61541521859507e-012\\
9	127.896078140961e-015\\
10	6.28823207310517e-015\\
11	310.273712088370e-018\\
12	15.3437212622757e-018\\
13	760.024772527884e-021\\
14	37.6976269610053e-021\\
15	1.87178593369082e-021\\
16	93.0024142963469e-024\\
17	4.62261144924141e-024\\
18	229.798404636386e-027\\
19	11.4244421256935e-027\\
20	567.993826421648e-030\\
21	28.2407178120988e-030\\
22	1.40421504952568e-030\\
23	69.8256601229396e-033\\
24	3.47226998617401e-033\\
25	172.671459957109e-036\\
26	8.58688791082188e-036\\
27	427.171581725612e-039\\
28	21.4345090941739e-039\\
29	2.27966882055557e-039\\
  };    
\addlegendentry{QN3};
\addplot [color=\QNthreeplotcolor,solid,mark=\QNthreeplot,mark options={solid}]
  table[row sep=crcr]{%
0	259.259545168478e-003\\
1	9.98235023835202e-003\\
2	334.534009707049e-006\\
3	13.9961070105821e-006\\
4	627.467741522253e-009\\
5	29.1708193455495e-009\\
6	1.39017324993530e-009\\
7	67.4353107514800e-012\\
8	3.31065239906249e-012\\
9	163.695241780531e-015\\
10	8.11741049425208e-015\\
11	402.569126823326e-018\\
12	19.9472043881884e-018\\
13	987.806452361823e-021\\
14	48.9256156200234e-021\\
15	2.42524404158591e-021\\
16	120.355390882836e-024\\
17	5.97940561690051e-024\\
18	297.319148162443e-027\\
19	14.7912412776711e-027\\
20	735.958285047761e-030\\
21	36.6150352941344e-030\\
22	1.82124926996500e-030\\
23	90.5688552866368e-033\\
24	4.50314899956444e-033\\
25	223.883564158719e-036\\
26	11.1309693669558e-036\\
27	553.480607951483e-039\\
28	27.4996196281083e-039\\
29	1.37495923815412e-039\\
  };   
\addlegendentry{QN4};
\addplot [color=\QNfourplotcolor,solid,mark=\QNfourplot,mark options={solid},mark repeat={0}]
  table[row sep=crcr]{%
0	259.259545168478e-003\\
1	24.8941240835058e-003\\
2	959.390668097561e-006\\
3	1.53308861385027e-006\\
4	3.92470438499221e-012\\
5	25.7209426576332e-024\\
6	6.00812618290976e-042\\
  };
\end{axis}
\end{tikzpicture}%
}
\subfigure[$r=100$]{
%
%
\begin{tikzpicture}
\begin{axis}[%
width=0.951\figurewidth,
height=\figureheight,
at={(0\figurewidth,0\figureheight)},
scale only axis,
unbounded coords=jump,
xmin=0,
xmax=30,
xlabel={Iteration $k$},
ymode=log,
ymin=1e-35,
ytick={1e2,1e-2,1e-6,1e-10,1e-14,1e-18,1e-22,1e-26,1e-30,1e-34},
xtick={10,20,30,40,50,60},
ymax=100,
grid,
yminorticks=true,
ylabel={Error in $(\mu_k,x_k)$},
axis background/.style={fill=white},
legend style={legend cell align=left,align=left,draw=white!15!black,at={(.95,0.95),anchor=north east}},
mark repeat={2}
]
  \addlegendentry{QN1};
\addplot [color=\QNoneplotcolor,solid,mark=\QNoneplot,mark options={solid}]
  table[row sep=crcr]{%
0	259.259545168478e-003\\
1	7.54334903675706e-003\\
2	576.677337760449e-006\\
3	44.0426967811988e-006\\
4	3.36393562076045e-006\\
5	256.932444405732e-009\\
6	19.6241300143015e-009\\
7	1.49886273335617e-009\\
8	114.480973044595e-012\\
9	8.74389154858286e-012\\
10	667.845820848630e-015\\
11	51.0091002325801e-015\\
12	3.89600148014907e-015\\
13	297.570971926863e-018\\
14	22.7280415022095e-018\\
15	1.73593501806057e-018\\
16	132.588212083121e-021\\
17	10.1268963414536e-021\\
18	773.477731536601e-024\\
19	59.0771131658618e-024\\
20	4.51222466751376e-024\\
21	344.637209894762e-027\\
22	26.3228928513196e-027\\
23	2.01050457747623e-027\\
24	153.559438884032e-030\\
25	11.7286483873496e-030\\
26	895.817241782556e-033\\
27	68.4212284637681e-033\\
28	5.22591473372376e-033\\
29	399.147834334763e-036\\
30      30.4863099395493e-036\\
  };    
\addlegendentry{QN2};
\addplot [color=\QNtwoplotcolor,solid,mark=\QNtwoplot,mark options={solid}]
  table[row sep=crcr]{%
0	259.259545168478e-003\\
1	7.54334903675706e-003\\
2	1.39468425360273e-006\\
3	11.7795033562986e-012\\
4	7.56170121798890e-015\\
5	340.187802435262e-021\\
6	125.106053539333e-024\\
7	10.0255161209891e-027\\
8	2.47195965580087e-030\\
9	304.614864996360e-036\\
10	53.2126481393879e-039\\
11	1.28577658090848e-039\\
12	1.28570794094873e-039\\
13	1.28570794094830e-039\\
14	1.28570794094830e-039\\
15	1.28570794094830e-039\\
16	1.28570794094830e-039\\
17	1.28570794094830e-039\\
18	1.28570794094830e-039\\
19	1.28570794094830e-039\\
20	1.28570794094830e-039\\
21	1.28570794094830e-039\\
22	1.28570794094830e-039\\
23	1.28570794094830e-039\\
24	1.28570794094830e-039\\
25	1.28570794094830e-039\\
26	1.28570794094830e-039\\
27	1.28570794094830e-039\\
28	1.28570794094830e-039\\
29	1.28570794094830e-039\\
  };    
\addlegendentry{QN3};
\addplot [color=\QNthreeplotcolor,solid,mark=\QNthreeplot,mark options={solid}]
  table[row sep=crcr]{%
0	259.259545168478e-003\\
1	7.42938441488100e-003\\
2	825.143399542806e-009\\
3	12.0418328346941e-012\\
4	7.54060179522387e-015\\
5	344.693270259621e-021\\
6	124.586512800289e-024\\
7	10.1111429054436e-027\\
8	2.45755979514331e-030\\
9	306.367148671685e-036\\
10	52.7628864650267e-039\\
11	184.785645513743e-042\\
12	184.591590589310e-042\\
13	184.591590568566e-042\\
14	184.591590568568e-042\\
15	184.591590568568e-042\\
16	184.591590568568e-042\\
17	184.591590568568e-042\\
18	184.591590568568e-042\\
19	184.591590568568e-042\\
20	184.591590568568e-042\\
21	184.591590568568e-042\\
22	184.591590568568e-042\\
23	184.591590568568e-042\\
24	184.591590568568e-042\\
25	184.591590568568e-042\\
26	184.591590568568e-042\\
27	184.591590568568e-042\\
28	184.591590568568e-042\\
29	184.591590568568e-042\\
  };    
\addlegendentry{QN4};
\addplot [color=\QNfourplotcolor,solid,mark=\QNfourplot,mark options={solid},mark repeat={0}]
  table[row sep=crcr]{%
0	259.259545168478e-003\\
1	100.098732158676e-006\\
2	100.098637483653e-012\\
3	100.098637483574e-024\\
4	27.8565196529942e-042\\
  };
\end{axis}
\end{tikzpicture}%
}
\caption{Convergence of the four different methods for the quadratic problem \eqref{eq:M_circle} when a) $r=0.5$, b) $r=100$.}
\label{fig:conv_poly}
\end{figure}

\gm{
Let $\rho_r$ be the spectral radius of the iteration matrix $B$ given in \gm{Corollary~\ref{cor:QNfactors}}, i.e.,
the convergence factor for 
QN2 and QN3
 (the residual inverse iteration).
In Figure~\ref{fig:rho_r} we
illustrate
 how $\rho_r$ behaves as a function of $r$, the
radius of the circle. 
The value of 
$\rho_r$ is computed for $10$ different values of $r$ varying from $10^{-1/2}$ to $10^5$.}
As expected from Corollary~\ref{Cor:clustering} we observe that $\rho_r \sim 1/r$. 

  \begin{figure}[h!] 
\begin{center}
\setlength\figureheight{4cm} 
\setlength\figurewidth{5.5cm}
%
%
\begin{tikzpicture}
\begin{axis}[%
width=0.951\figurewidth,
height=\figureheight,
at={(0\figurewidth,0\figureheight)},
scale only axis,
unbounded coords=jump,
xmin=-2,
xmax=6,
xlabel={$\ln(r)$},
ymode=log,
ymin=1e-9,
ytick={1e2,1e0,1e-2,1e-4,1e-6,1e-8},
ymax=1e3,
grid,
ylabel={$\rho_r$},
legend style={legend cell align=left,align=left,draw=white!15!black,at={(.95,0.95),anchor=north east}}
]
\addlegendentry{$\rho_r$ for QN2 and QN3 };
\addplot [color=red,solid,mark=o,mark options={solid}]
  table[row sep=crcr]{%
   -1.00000000000000e+000    18.7964887882518e-003\\
   -333.333333333333e-003    3.55282009758731e-003\\
    333.333333333333e-003    765.185411881937e-006\\
    1.00000000000000e+000    165.081531917725e-006\\
    1.66666666666667e+000    35.5786645399032e-006\\
    2.33333333333333e+000    7.66581461443115e-006\\
    3.00000000000000e+000    1.65157887993336e-006\\
    3.66666666666667e+000    355.823240319143e-009\\
    4.33333333333333e+000    76.6598562607189e-009\\
    5.00000000000000e+000    16.5158682912017e-009\\
  };
  \addlegendentry{reference slope $-1$};
\addplot [color=blue,solid,mark=x,mark options={solid}]
  table[row sep=crcr]{%
   -1.00000000000000e+000    1.87964887882518e+000\\
   -333.333333333333e-003    404.958074962050e-003\\
    333.333333333333e-003    87.2455724706773e-003\\
    1.00000000000000e+000    18.7964887882518e-003\\
    1.66666666666667e+000    4.04958074962050e-003\\
    2.33333333333333e+000    872.455724706772e-006\\
    3.00000000000000e+000    187.964887882518e-006\\
    3.66666666666667e+000    40.4958074962050e-006\\
    4.33333333333333e+000    8.72455724706773e-006\\
    5.00000000000000e+000    1.87964887882518e-006\\
  };
\end{axis}
\end{tikzpicture}%
\end{center}
\caption{Plot of the convergence factor $\rho_r$ versus $r$ for QN2 and the quadratic problem
\eqref{eq:M_circle}.}
\label{fig:rho_r}
\end{figure}
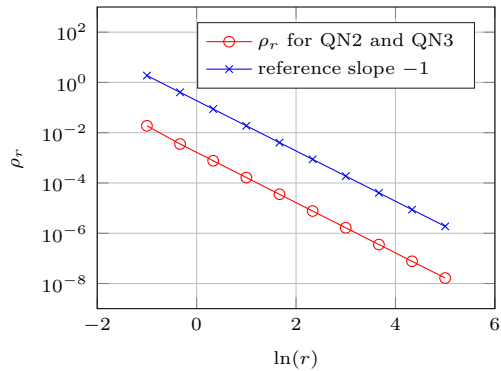

\subsection{A large scale problem}
 
We now consider a
large--scale NEP 
which arises in the study of waves traveling
in a periodic medium
\cite{WAVEGUIDE_ARNOLDI_2015,Tausch:2000:WAVEGUIDE}.
\gm{More precisely, the waveguide eigenvalue problem, without the Cayley transformation,
associated to 
the waveguide described in
\cite[Section 5.2]{WAVEGUIDE_ARNOLDI_2015} 
is considered. 
The problem is formulated as}
\begin{align*}
 M(\lambda) = 
 \begin{pmatrix}
  Q(\lambda)		&	C_1(\lambda)	\\
  C_2^T			&	P(\lambda)
 \end{pmatrix}
 .
\end{align*}
We \ak{choose} the discretization parameters $n_x=200$ and $n_z=201$, which means that the size of the NEP is $n=n_x n_z+2 n_z=40602$. 
The matrix $C_2^T$ and the second degree polynomials $Q(\lambda)$ and $C_1(\lambda)$ are \ak{sparse, and the matrix} $P(\lambda)$ is dense and it 
is defined by nonlinear functions of $\lambda$ involving 
square roots of polynomials. 
\ak{The matrix--vector product} 
$P(\lambda) w$ is efficiently computed using two Fast Fourier Transforms (FFTs) and a multiplication with a diagonal matrix. 
The linear systems involving the matrix $M(\sigma)$ can be solved by precomputing a Schur complement. See \cite{WAVEGUIDE_ARNOLDI_2015} for a full description of the problem.

We \ak{compare} \gm{QN1, QN2 and QN3} (residual inverse iteration) for approximating 
a specific eigenpair. 
\gm{The pair $(\lambda,v)$ denotes the accurate approximation of the wanted eigenpair. }
\ej{The shift $\sigma$ is selected} close to the wanted eigenvalue $\lambda$, \ak{more precisely, $| \sigma - \lambda | \approx 0.42$ (see Figure~\ref{FIG:waveguide_spectrum}),
and the initial guess $x_0$ is selected such that $\| x_0 - v \| \approx 10^{-5}$.} 
The error is computed as the absolute value of the 
distance between the wanted eigenvalue and the current eigenvalue approximation, 
namely $| \mu_k - \lambda |$. 
If an initial guess of the eigenvector is provided, 
all the methods present 
similar convergence rate and \gm{QN2} and \gm{QN3} are slightly faster then \gm{QN1}
\ak{(see Figure~\ref{FIG:waveguideQN123}.)}
\gm{
The block $(1,2)$ of the matrix $A_1$ defined in \eqref{eq:A1def} has norm approximatively $10^{-3}$. 
This may suggest that the spectral radius of this matrix is close to the spectral radius of $B$ given in \eqref{eq:EliasConvRate}.
In particular the convergence rate of  
\gm{QN1} is expected to be close to the convergence rate of \gm{QN2} and \gm{QN3}. This is consistent with the numerical simulation.}
With a random initial guess $x_0$ of the eigenvector, 
\gm{QN1} does not converge 
whereas \gm{QN2} and \gm{QN3} still converge with the same convergence rate but with they require more iterations since the initial error is larger. See Figure~\ref{FIG:waveguideQN123}.
This can be justified observing that, in this case, the block $(1,2)$ of the matrix $A_1$ defined in \eqref{eq:A1def} has norm approximatively $10^{-1}$. 
\gm{Therefore, the convergence factor of QN1 is expected to be significantly different from QN2 and QN3 by using the previous reasoning.} 

 \begin{figure} 
 \begin{minipage}[b]{\textwidth}
	\centering
	\setlength\figureheight{4cm} 
	\setlength\figurewidth{10cm}
%
%
\definecolor{mycolor1}{rgb}{1.00000,0.00000,1.00000}%
\begin{tikzpicture}

\begin{axis}[%
width=0.951\figurewidth,
height=\figureheight,
at={(0\figurewidth,0\figureheight)},
scale only axis,
xmin=-4.5,
xmax=0.1,
ymin=-7,
ymax=0.3,
xtick={-1,-2,-3,-4,-5},
axis lines=center, 
axis background/.style={fill=white},
legend style={at={(0.4,0.27)},anchor=north,legend cell align=left,align=left,draw=white!15!black}
]
\addplot [color=black,only marks,mark=\eigplot,mark options={solid}]
  table[row sep=crcr]{%
-3.18753056861896	-3.14231088534989\\
-1.34903031021354	-1.78407133228711\\
-1.34620888618634	-4.50007324031644\\
-0.559230036676164	-2.17853058817262\\
-0.556643685387807	-4.11369935864428\\
-0.000461525568890464	-3.78060256693498\\
-0.000756308592683583	-2.51619891702154\\
-1.15947397845453	-0.225620833634788\\
-1.16037030055377	-6.06141023171197\\
-4.26799400753975	-0.261146693581903\\
-4.26701951355202	-6.02603392812274\\
-0.50220495759151	-0.371726054947731\\
-0.500836411918141	-5.92057022297237\\
};
\addlegendentry{~eigenvalues};

\addplot [color=red,only marks,mark=o,mark options={solid}]
  table[row sep=crcr]{%
-1.34620888618634	-4.50007324031644\\
};
\addlegendentry{~computed $\lambda$};

\addplot [color=mycolor1,only marks,mark=\shiftplot,mark options={solid}]
  table[row sep=crcr]{%
-1.64620888618634	-4.20007324031644\\
};
\addlegendentry{~shift $\sigma$};

\end{axis}
\end{tikzpicture}%
	\caption{Spectrum of the waveguide eigenvalue problem described in 
		  \cite[Section 5.2]{WAVEGUIDE_ARNOLDI_2015}}
	\label{FIG:waveguide_spectrum}
\end{minipage} \\[1cm]
\begin{minipage}[b]{\textwidth}
	\centering
	\setlength\figureheight{4cm} 
	\setlength\figurewidth{10cm}
	\input{ waveguide_QN123_conv.tikz}
	\caption{Convergence of QN1, QN2 and QN3 for computing 
	the eigenvalue \ej{$\lambda$}.}	 \label{FIG:waveguideQN123}
\end{minipage}
\end{figure}

\section{Conclusions}

We have here presented four iterative methods and showed
how they can be analyzed with techniques of quasi-Newton
methods, and \ak{Keldysh's theorem}. We have also
illustrated how \ak{two} well-established methods
can be interpreted as quasi-Newton methods.
\gm{
A secondary conclusion that we wish to stress
here is that, in general, methods of this type \ej{can in an insightful way be analyzed} in
the framework of quasi-Newton methods.}
We have illustrated
how linear, as well as higher order convergence,
can be approached with quasi-Newton results. There
are many results for quasi-Newton methods,
which can potentially be applied to these methods, e.g.,
techniques which improve the convergence basin \cite{eisenstat1994globally,gonzalez2009newton},
scaling techniques \cite{deuflhard2011newton}, 
adaption for non-smooth (or almost non-smooth) problems 
\cite{Mietanski:2011:INEXACT} and other convergence results
\cite{ypma1984local}.

The method corresponding to keeping only the (1,1)-block constant,
i.e., \eqref{eq:QN2}, is to our knowledge new method in the context
of NEPs. Moreover, although it has several similarities
with residual inverse iteration in terms of convergence
and behavior for linear problems, it can be more attractive
than residual inverse iteration. Note that unlike QN2,
residual inverse iteration requires the
solution to a nonlinear scalar problem. For some NEPs,
the action of $M(\lambda)$ is only implicitly
available, e.g., in the form
of a differential equation as in \cite{Rott:2010:ITERATIVE}.
\gm{Therefore,}
the computation of a solution to the scalar nonlinear
equation \eqref{eq:rf} is computationally more demanding
than computing the product $M'(\mu)r$ which is required in \eqref{eq:zkdef}.

We do have a negative conclusion in this paper. We have concluded that QN1
does not appear very competitive in practice since better convergence
is usually achieved from QN2, although QN1 is based on the most common quasi-Newton approach,
i.e., keeping the \gm{Jacobian matrix} constant.
\ej{There is another important Newton-like algorithm
which we have not considered in this manuscript,
the Jacobi-Davidson algorithm (\cite{Schreiber:2008:PHD,Sleijpen:1996:POLYNOMIAL}). Although
the Jacobi-Davidson algorithm does
have interpretations in terms of Newton's method, as
e.g., pointed out in \cite[Section~6]{sleijpen2006jacobi},
our results here are not directly applicable. In \cite[Section~6]{sleijpen2006jacobi} the authors point out that the correction equation
of the Jacobi-Davidson algorithm can be derived from Newton's method
on the nonlinear equation $G(x)=M(p(x))x$ where $p(x)$ is the Rayleigh functional. Note that this nonlinear equation is different from
our augmented system \eqref{eq:augsys}. Moreover, since
the solution is a manifold and the solution (eigenvector)
is not isolated in the standard sense which prevents us from directly
applying results for quasi-Newton methods.}

Finally, we wish to specifically stress that the value of the presented results
may be of interested considerably beyond the scope of the presented
methods.  Several of the methods presented here form the basis
of other state-of-the-art such as the subspace
accelerated extensions of residual inverse iteration (the nonlinear
Arnoldi method \cite{Voss:2004:ARNOLDI}), preconditioned versions \cite{Effenberger:2014:RESINV}
and inner-outer-iteration constructions as in \cite{Xue:2011:EFFICIENT_TR}.
Those extensions may possibly also be interpreted in a quasi-Newton
setting, although the extension
 requires attention beyond this manuscript.
\section*{Acknowledgements}
We thank Wim Michiels (KU Leuven) for valuable
discussions regarding partial fraction expansions for 
time-delay systems.
\bibliographystyle{plain}
\bibliography{eliasbib,misc,functions-of-matrices}

\end{document}